\documentclass[a4paper,11pt]{article}

\usepackage{amsfonts}
\usepackage{epsf,latexsym,amsfonts,amsbsy,mathrsfs}
\usepackage{cite}
\usepackage{amsmath, amssymb, amsthm,amscd,amsxtra}
\usepackage[]{graphicx,subfigure}
\usepackage{float}
\usepackage{multirow}
\usepackage{bbm}
\usepackage{stmaryrd}
\usepackage{epsfig}

\newtheorem{theorem}{Theorem}[section]
\newtheorem{lemma}[theorem]{Lemma}

\newtheorem{algorithm}{Algorithm}[section]

\theoremstyle{definition}

\theoremstyle{remark}
\newtheorem{remark}[theorem]{Remark}

\numberwithin{equation}{section}

\newcommand{\ii}{\rm i}

\newcommand{\nn}{\nonumber}

\newcommand{\Bn}{{\boldsymbol{n}}}

\newcommand{\Bq}{{\boldsymbol{q}}}
\newcommand{\Br}{{\boldsymbol{r}}}

\newcommand{\Bv}{{\boldsymbol{v}}}
\newcommand{\Bw}{{\boldsymbol{w}}}
\newcommand{\Bx}{{\boldsymbol{x}}}

\newcommand{\BA}{{\boldsymbol{A}}}

\newcommand{\BD}{{\boldsymbol{D}}}

\newcommand{\BI}{{\boldsymbol{I}}}

\newcommand{\BL}{{\boldsymbol{L}}}
\newcommand{\BM}{{\boldsymbol{M}}}

\newcommand{\BQ}{{\boldsymbol{Q}}}
\newcommand{\BR}{{\boldsymbol{R}}}
\newcommand{\BS}{{\boldsymbol{S}}}

\newcommand{\BU}{{\boldsymbol{U}}}
\newcommand{\BV}{{\boldsymbol{V}}}

\newcommand{\Cg}{{\mathcal G}}

\newcommand{\Ct}{{\mathcal T}}

\newtheorem{exm}{Example}[section]

\title{A Robust Multilevel Method for Hybridizable
Discontinuous Galerkin Method for the Helmholtz Equation}

\author{Huangxin Chen\thanks{School of Mathematical Sciences, Xiamen University, Xiamen, 361005, People's Republic of China ({\tt chx@xmu.edu.cn}).}~~, Peipei Lu$^\dag$, and Xuejun Xu\thanks{LSEC, Institute of
Computational Mathematics and Scientific/Engineering Computing,
Academy of Mathematics and System Sciences, Chinese Academy of
Sciences, P.O.Box 2719, Beijing, 100190, People's Republic of China ({\tt lupeipei@lsec.cc.ac.cn, xxj@lsec.cc.ac.cn}).}}

\date{}
\begin{document}
\maketitle
\begin{abstract}
A robust multilevel preconditioner based on the hybridizable
discontinuous Galerkin method for the Helmholtz equation with high wave number is
presented in this paper. There are two keys in our algorithm, one
is how to choose a suitable intergrid transfer operator, and the other is
using GMRES smoothing on coarse grids. The multilevel method is
performed as a preconditioner in the outer GMRES iteration. To give
a quantitative insight of our algorithm, we use local Fourier analysis
to analyze the convergence property of the proposed multilevel
method. Numerical results show that for
fixed wave number, the convergence of the algorithm is mesh
independent. Moreover, the performance of the algorithm depends
relatively mildly on wave number.
\end{abstract}

\section{Introduction}
In this paper we consider the Helmholtz equation with Robin boundary
condition which is the first order approximation of the radiation
condition. The equation is written in a mixed form as follows: Find
$({\Bq},u)$ such that
\begin{align}
\label{HS}
{\bf i}\kappa{\Bq}+\nabla u&=0 \qquad {\rm in }\ \Omega,\\
\label{HSE}{\bf i}\kappa u+{\rm div}\, {\Bq}&=f  \qquad {\rm in }\ \Omega,\\
\label{HE} -{\Bq}\cdot {\Bn}+u&=g \qquad {\rm on }\
\partial\Omega,
\end{align}
where $\Omega \subset \mathbb R^d$, $d=2,3$, is a polygonal or
polyhedral domain, $\kappa>0$ is known as the wave number, ${\bf
i}=\sqrt{-1} $ denotes the imaginary unit, and ${\Bn}$ denotes the
unit outward normal to $\partial \Omega$.  Helmholtz equation finds
applications in many important fields, e.g., in acoustics, seismic
inversion and electromagnetic, but how to solve the Helmholtz
equation efficiently is still of great challenge.

The strong indefiniteness has prevented the standard multigrid
methods from being directly applied to the discrete Helmholtz
equation. In \cite{EEO01}, Elman, Ernst and O'Leary modified the
standard multigrid algorithm by adding GMRES iterations as corrections on coarse
grids and using it as an outer iteration. But in order to obtain
a satisfactory convergence behavior, a relatively large number of GMRES
smoothing should be performed on coarse grids which leads to
relatively large memory requirement, so the optimality of the multigrid algorithm  cannot be guaranteed. In \cite{chen}, the authors
utilized the continuous interior penalty finite element methods
 \cite{Wu12,Wu12-hp} to construct the stable coarse grid
correction problems, which reduces the steps of GMRES smoothing on
coarse grids. Based on the fact that the error components which
cannot be reduced by the standard multigrid can be factorized by
representing it as the product of a certain high-frequency Fourier
component and a ray function, Brandt and Livshits introduced so-called
wave-ray multigrid methods in \cite{Brandt97,Livshits}. Although
this method exhibits high convergence rate with increasing wave number,
it does not easily generalize to unstructured grids and complicated
Helmholtz problems. Besides, shifted Laplacian preconditioners
\cite{Erlangga04,Erlangga06} and sweeping preconditioners
\cite{Engquist1,Engquist2} based on an approximate $LDL^t$
factorization were introduced to solve the Helmholtz equation with high wave number. A survey of the development of fast
iterative solvers  can be found in \cite{Erlangga08,Ernst11}.

Hybridizable discontinuous Galerkin (HDG) method has two main
advantages in the discretization of Helmholtz equation. First, it is
a stable method, which means that the discrete system is always well-posed without any mesh
constraint. Rigorous convergence analysis of the HDG method for
Helmholtz equation can be found in \cite{HDG}. Second, comparing to
standard discontinuous Galerkin method, HDG method results in
significantly reducing the degrees of freedom, especially when the
polynomial degree $p$ is large. However, to the best of our
knowledge, no efficient iterative method or preconditioner for HDG
discretization system for the Hemholtz equation in the literature
has been proposed.

The hybridized system is a linear equation for Lagrange multipliers
which is obtained by eliminating the flux as well as the primal
variable. For the second-order elliptic problems, a Schwarz
preconditioner  for the algebraic system was presented in \cite{JG}. In \cite{J-S}, the authors
consider the application of a variable V-cycle multigrid algorithm
for the hybridized mixed method for second-order elliptic boundary
value problems. In their multigrid algorithm, both smoothing
and correction on coarse grids are based on standard piecewise linear
continuous finite element discretization system. The convergence of
the multigrid algorithm is dependent on an assumption that the number of
smoothings increases in a specific way (see Theorem 3.1 in
\cite{J-S} for details). The critical ingredient in the algorithm is
how to choose a suitable intergrid transfer operator. Numerical experiments in \cite{J-S}
show that certain `obvious' transfer operators lead to slow
convergence.

The objective of this paper is to propose a robust multilevel method for the HDG method approximation of the
Helmholtz equation. The main ingredients
in multilevel method are how to construct coarse grid correction
problem and perform efficient smoothing. Since strong indefiniteness
arises for Helmholtz equation with large wave number, standard
Jacobi or Gauss-Seidel smoothers become unstable on the coarse
grids. Motivated by the idea in \cite{EEO01}, we use GMRES smoothing
for those coarse grids. Unlike the smoothing strategy in
\cite{EEO01}, the number of GMRES smoothing steps in our algorithm
is much smaller, even if one smoothing step may guarantee the convergence of our multilevel algorithm. Moreover, both smoothing on fine and coarse grids
in our multilevel method are based on hybridized system of Lagrange
multiplier on each level.

{\it Local Fourier analysis} (LFA) has been introduced for multigrid
analysis by Achi Brandt in 1977 (cf. \cite{Brandt77}). We mainly
utilize the LFA to analyze smoothing properties of relaxations and
convergence properties of two  and three level methods in the one
dimensional case. This may provide quantitative insights into the
proposed multilevel method for Helmholtz problem
(\ref{HS})-(\ref{HE}). A survey for LFA can be found in
\cite{WJ04}.

The remainder of this paper is organized as follows:  In section 2, we firstly
review the formulation of HDG method for the Helmholtz equation and
present our multilevel algorithm. The stability
estimate of the intergrid transfer operator will be carried out in section 3. Section 4
is devoted to the LFA of the multilevel method in one dimensional
case. Finally, we give some numerical results to
demonstrate the performance of our multilevel method.

\section{ HDG method and its multilevel algorithm}\label{notation}
Let $\mathcal T_h$ be a quasi-uniform subdivision of
$\Omega$, and denote the collection of edges (faces) by $\mathcal
E_h$, while the set of interior edges (faces) by $\mathcal E_h^0$ and
the collection of element boundaries by $\partial \mathcal
T_h:=\{\partial T| T\in \mathcal T_h\}$.
We define $h_T:={\rm diam}(T)$ and let $h:={\rm max}_{T\in \Ct_h}h_T$.
Throughout this paper we
use the standard notations and definitions for Sobolev spaces (see,
e.g., Adams\cite{Adams}).

On each element $T$ and each edge (face) $F$, we define the local
spaces of polynomials of degree $p\geq 1$:
$${\BV}(T):=(\mathcal P_p(T))^d,\qquad W(T):=\mathcal P_p(T),\qquad M(F):=\mathcal P_p(F),$$
where $\mathcal P_p(S),S=T$ or $F$, denotes the space of polynomials
of total degree at most $p$ on $S$. The corresponding global finite
element spaces are given by
\begin{align*}
{\BV}_h^p:&=\{{\Bv}\in {\BL}^2(\Omega)\  |\ {\Bv}|_T\in {\BV}(T)\ {\rm for\  all }\ T\in \mathcal T_h\},\\
W_h^p:&=\{w\in L^2(\Omega)\ |\ w|_T\in W(T)\ {\rm for\ all }\ T\in \mathcal T_h\},\\
M_h^p:&=\{ \mu\in L^2(\mathcal E_h)\ |\  \mu|_F\in M(F)\ {\rm for\ all }\ F\in \mathcal E_h\},
\end{align*}
where ${\BL}^2(\Omega):=( L^2(\Omega))^d$, $L^2(\mathcal
E_h):=\Pi_{F\in \mathcal E_h}L^2(F)$. On these spaces we define the
bilinear forms
\begin{align*}
({\Bv},{\Bw})_{\mathcal T_h}:=\sum_{T\in \mathcal
T_h}({\Bv},{\Bw})_{T}, \ (v,w)_{\mathcal T_h}:=\sum_{T\in \mathcal
T_h}(v,w)_T, \ {\rm and} \ \langle v,w \rangle_{\partial \mathcal
T_h}:=\sum_{T\in \mathcal T_h}\langle v,w\rangle_{\partial T},
\end{align*}
with $({\Bv},{\Bw})_{T}:=\int_T {\Bv}\cdot {\Bw}dx$,
$(v,w)_T:=\int_T { v} { w}dx$ and $\langle v,w\rangle_{\partial
T}:=\int_{\partial T} { v} { w}ds$.

The HDG method yields finite element approximations $({\Bq}_h,
u_h,\hat u_h)\in {\BV}_h^p\times W_h^p\times M_h^p$ which satisfy
\begin{align}
\label{P1}
({\bf i}\kappa {\Bq}_h,\overline{\Br})_{\mathcal T_h}-(u_h, \overline{{\rm div \, {\Br}}})_{\mathcal T_h}+\langle \hat u_h,\overline {{\Br}\cdot {\Bn}}\rangle_{\partial \mathcal T_h}&=0,\\
\label{P2}
({\bf i}\kappa u_h,\overline w)_{\mathcal T_h}-({\Bq}_h,\overline{\nabla w})_{\mathcal T_h}+\langle \hat {\Bq}_h\cdot {\Bn} ,\overline w\rangle_{\partial \mathcal T_h}&=(f,\overline w)_{\mathcal T_h} ,\\
\label{P3}
\langle -\hat{\Bq}_h\cdot {\Bn}+ \hat u_h,\overline \mu\rangle_{\partial \Omega}&=\langle g, \overline \mu\rangle_{\partial \Omega},\\
\label{P4} \langle \hat{\Bq}_h\cdot {\Bn},\overline
\mu\rangle_{\partial \mathcal T_h \backslash
 \partial \Omega}&=0,
\end{align}
for all ${\Br}\in {\BV}_h^p$, $w\in W_h^p$, and $\mu\in M_h^p$,
where the overbar denotes complex conjugation. The numerical flux
$\hat {\Bq}_h$ is given by
\begin{align}\label{numerical-flux}
\hat {\Bq}_h={\Bq}_h+\tau_h (u_h-\hat u_h){\Bn} \qquad {\rm on } \
\partial \mathcal T_h,
\end{align}
where the parameter $\tau_h$ is the so-called {\it local
stabilization parameter} which has an important effect on both the
stability of the solution and the accuracy of the HDG scheme. Let
$\tau_{h,T}$ be the value of $\tau_h$ on the element $T$. We always
choose $\tau_{h,T} = \frac{p}{\kappa h_T}$. One of the advantages of
HDG methods is the elimination of both ${\Bq}_h$ and $u_h$ from the
equation, and then we may obtain a formulation in terms of $\hat u_h$ only. Next
we define the discrete solutions of the local problems: For each
function $\lambda\in M_h^p$, $({\bf \mathcal Q}_\lambda, \mathcal
{U}_\lambda)\in {\BV}(T)\times W(T)$ satisfies the following
formulation
\begin{align}
\label{local1} ({\bf i}\kappa\mathcal
{Q}_\lambda,\overline{{\Br}})_T-(\mathcal {U}_\lambda,\overline{{\rm
div}\,{\Br}})_T&=-\langle \lambda,
\overline{{\Br}\cdot {\Bn}}\rangle_{\partial T}, \quad \forall {\Br}\in {\BV}(T),\\
\label{local2} ({\bf i}\kappa \mathcal
{U}_\lambda,\overline{w})_T-(\mathcal {Q}_\lambda,\overline{\nabla
w})_T+\langle\hat {\mathcal {Q}}_\lambda\cdot
{\Bn},\overline{w}\rangle_{\partial T}&=0,\quad  \forall
{w}\in W(T),
\end{align}
where $\hat {\mathcal {Q}}_\lambda\cdot {\Bn}=\mathcal {Q}_\lambda\cdot
{\Bn}+\tau_h(\mathcal {U}_\lambda-\lambda)$.  For $f\in L^2(\Omega)$, $({\bf
\mathcal Q}_f, \mathcal {U}_f)\in {\BV}(T)\times W(T)$ is defined as
follows:
\begin{align}
\label{local3}
({\bf i}\kappa\mathcal {Q}_f,\overline{{\Br}})_T-(\mathcal {U}_f,\overline{{\rm div}\,{\Br}})_T&=0 ,\quad  \forall {\Br}\in {\BV}(T),\\
\label{local4} ({\bf i}\kappa \mathcal
{U}_f,\overline{w})_T-(\mathcal {Q}_f,\overline{\nabla
w})_T+\langle\hat {\mathcal {Q}}_f\cdot
{\Bn},\overline{w}\rangle_{\partial T}&=(f,\overline{w})_T,\quad  \forall {w}\in W(T),
\end{align}
where $\hat {\mathcal {Q}}_f\cdot {\Bn}=\mathcal {Q}_f\cdot
{\Bn}+\tau_h\mathcal {U}_f$. Then $\hat u_h$ is the solution of the following equation
\begin{align}
\label{AS}
a_h(\hat u_h,\mu)=b_h(\mu),\quad \forall  \mu\in M_h^p,
\end{align}
where
\begin{align}
\label{DF}
a_h(\lambda,\mu):=-\langle\hat {\mathcal {Q}}_\lambda,\overline{\mu}\rangle_{\partial\mathcal T_h}+\langle \lambda,\overline{\mu}\rangle_{\partial \Omega},
\end{align}
\begin{align*}
b_h(\mu):=\langle\hat {\mathcal
{Q}}_f,\overline{\mu}\rangle_{\partial\mathcal
T_h}+\langle g,\overline{\mu}\rangle_{\partial \Omega}.
\end{align*}
We focus on designing a multilevel method for the linear algebraic system
(\ref{AS}).

Let $\{\mathcal {T}_{l}\}_{l=0}^L$ be a shape regular family of
nested conforming  triangulations of $\Omega$, which means that $\mathcal
{T}_{0}$ is a quasi-uniform initial mesh and $\mathcal
{T}_{l}$ is obtained by  quasi-uniform
refinement of $\mathcal
{T}_{l-1},l\geq 1$. For simplicity, we denote
by $a_l(\cdot,\cdot)$ the bilinear form $a_{h_l}(\cdot,\cdot)$ on
$M_l$, where $h_l$ is the mesh size of $\Ct_l$, meanwhile we denote by $M_l$ for the $hp$-HDG approximation space
$M^p_{h_l}$ on $\Ct_l$,  the collection of edges of $\Ct_l$ is
denoted by $\mathcal E_l$.  Let $I_l:
M_l\rightarrow M_L$ be the intergrid transform operator, which will
be specified later. Define projections $P_l$, $Q_l$ : $M_L
\rightarrow M_l$ as
$$
a_l(P_l v, w) = a_L(v,  I_lw)   , \quad \langle Q_lv,\overline{w}\rangle_{\partial \Ct_l} = \langle v,
\overline{I_lw}\rangle_{\partial \Ct_L}  , \quad v\in M_L,\, w \in M_l.
$$
The existence and uniqueness solution of problem (\ref{AS})
imply the well-posedness of the above definition. For $0\leq l \leq L$,
define $A_l: M_l \rightarrow M_l$, $F_l\in M_l$ by means of
\begin{align}
\langle A_l v,\overline{w}\rangle_{\partial \Ct_l} = a_l(v,w),\   \langle F_l,\overline{w}\rangle_{\partial \Ct_l} = b_l(w) ,
\quad v,w\in M_l.\label{Ah}
\end{align}
Let $R_l:M_l\rightarrow M_l$ be the smoothing operator on $M_l$
which is chosen as weighted Jacobi or Gauss-Seidel relaxation. In
fact, both weighted Jacobi and Gauss-Seidel relaxation can be
used on the fine grids. Otherwise, we choose GMRES relaxation as a
smoother, and we will give some illustration in Section 4. Now we
 state our multilevel method.

\begin{algorithm}\label{alg}
Given an arbitrarily chosen initial iterate $u^0\in M_L$, we seek
$u^n\in M_L$ as follows:

\smallskip

Let $v_0 = u^{n-1}$. For $l=0,1,\cdots,L$, compute $v_{l+1}$ by

\smallskip

$1)$ When $l=0$, $v_1=v_0+\mu_0I_0(A_0)^{-1} Q_0 (F_L - A_L v_0)$.
For $l=1,\cdots,L$, if $\kappa h_l / p \geq \alpha$, perform $m_1$
steps {\rm GMRES} smoothing for the correction problem $A_l w_l =
Q_l(F_L - A_L v_l)$, and set
\[
v_{l+1} = v_l + \mu_l I_lw_l,
\]
else perform $m_2$ steps of weighted Jacobi relaxation $R^{J}_{l}$
or Gauss-Seidel relaxation $R^{GS}_{l}$,
\[
v_{l+1} = v_l + \mu_l I_l(R_{l})^{m_2}Q_l(F_L - A_L v_l),
\]
where $R_{l}=R^{J}_{l}$ or $R^{GS}_{l}$. We will always choose the
parameters $\alpha$ and $\{\mu_l\}^{L}_{l=0}$ as $0.5$ in this
paper.

\smallskip

$2)$ For $l=L,\cdots,1$, if $\kappa h_l / p < \alpha$, perform $m_3$
steps of $R_{l} = R^J_{l}$ or $R^{GS}_{l}$ to obtain $v_{2L+2-l}$,
\[
v_{2L+2-l} = v_{2L+1-l} + \mu_l I_l(R_{l})^{m_3}  Q_l(F_L -
A_Lv_{2L+1-l}) ; \label{smooth-fem-standard}
\]
else perform $m_4$ steps of {\rm GMRES} smoothing for the correction
problem $A_l w_l = Q_l(F_L - A_L v_{2L+1-l})$, and set
\[
v_{2L+2-l} = v_{2L+1-l} + \mu_l I_lw_l. \label{smooth-cip-gmres}
\]
When $l=0$, $v_{2L+2} = v_{2L+1}+  \mu_0I_0(A_0)^{-1} Q_0 (F_L - A_L
v_{2L+1})$.

\smallskip

$3)$ Set $u^n = v_{2L+2} $.
\end{algorithm}

At the end of this section, we give the definition of the  transfer
operator $I_l: M_l\rightarrow M_L,0\leq l\leq L-1$. Note that $I_L: M_L\rightarrow M_L$ is the identity operator. Denote
$W_l^c:=\{v\in C(\Omega)\ | \ v|_T\in W(T),\ \forall T\in
\mathcal T_l\}$. We first define an auxiliary operator
$I_l^W: M_l\rightarrow W_l^c,0\leq l\leq L-1$.

Case 1: $p=1$. Let
\begin{displaymath}
I_l^W\lambda(z_n)= \frac{\sum_{e\subset \mathcal
E_l(z_n)}\lambda_e(z_n)}{|\mathcal E_l(z_n)|}\quad  \forall z_n\in
\mathcal V_l,
\end{displaymath}
where $\mathcal V_l$ is the collection of the vertices of $\Ct_l$. For any $z_n\in
\mathcal V_l$,
$\mathcal E_l(z_n)$ is the collection of edges in $\mathcal E_l$ which contain $z_n$, while $|\mathcal E_l(z_n)|$ is the number of edges  in $\mathcal E_l(z_n)$.

Case 2: $p=2$. Let
\begin{displaymath}
I_l^W\lambda(z_n)=\left\{\begin{array}{ll}
\frac{\sum_{e\subset \mathcal E_l(z_n)}\lambda_e(z_n)}{|\mathcal E_l(z_n)|}\ & z_n\in \mathcal V_l;\\
\lambda (z_n)& z_n \in \mathcal N_l\backslash  \mathcal V_l,
\end{array} \right.\\
\end{displaymath}
where $\mathcal N_l$ is the degree of freedom of the space $W_l^c$.

Case 3: $p\geq 3$. Let
\begin{displaymath}
I_l^W\lambda(z_n)=\left\{\begin{array}{ll}
\frac{\sum_{e\subset \mathcal E_l(z_n)}\lambda_e(z_n)}{|\mathcal E_l(z_n)|}\ & z_n\in \mathcal V_l;\\
\lambda (z_n)& z_n \in \mathcal N_l\backslash  (\mathcal V_l\cup \mathcal N_l^0);\\
U_{\lambda} (z_n)& z_n \in \mathcal N_l^0,
\end{array} \right.\\
\end{displaymath}
where $\mathcal N_l^0$ is the degree of freedom in the interior of
every element $T\in\Ct_l$. For the $p\geq 3$ case, $\mathcal N_l^0$
is not an empty set, but the space $M_l$ dose not provide any
information for the  degree of freedom  in the interior of $T$.
Hence we use the solution of the local problem
(\ref{local1}-\ref{local2}) to define it. Note that this procedure only
involves the computation of the local problems and can be parallel
implemented.

With the help of the above operator $I_l^W$, we may define $I_l:
M_l\rightarrow M_L,0\leq l\leq L-1$ as follows:
\begin{align}
\label{trans}
I_l\mu|_{\mathcal E_L}:=I_l^W\mu|_{\mathcal E_L}.
\end{align}
Throughout this paper, we use notations $A\lesssim B$ and $A\gtrsim
B$ for the inequalities $A\leq CB$ and $A\geq CB$, where $C$ is a
positive number independent of the mesh sizes and mesh levels.

\section{The stability of the intergrid transfer operator}
The design of the stable intergrid transfer operator is critical for the success of the nonnested multilevel method.
The failure of certain `obvious' transfer operators in \cite{J-S} is due to the fact that
the energy error increases by using these operators. In the following, we will analyze the stability estimate of our intergrid transfer operator in the energy norm. From the numerical results, we may find that this intergrid transfer operator works well in our multilevel method for the Helmholtz problem. Consider the Possion equation:
\begin{align}
\label{4Lap-2}
-\triangle U=f\qquad {\rm {in} }\ \Omega,\\
\label{4Lap-2-b}
U=g\qquad {\rm {on} }\ \partial \Omega,
\end{align}
where $f\in L^2(\Omega)$ and $g\in L^2(\partial\Omega)$. Clearly,
(\ref{4Lap-2}-\ref{4Lap-2-b}) can be rewritten in a mixed form as
finding $({\BQ},U)$ such that
\begin{align}
\label{4Lap-1-1}
\BQ =-\nabla U \qquad {\rm {in} }\ \Omega,\\
\label{4Lap-1-2}
{\rm div}\ {\BQ} =f \qquad {\rm {in} }\ \Omega,\\
\label{4Lap-1-3}
U=g\qquad {\rm {on} }\ \partial \Omega.
\end{align}

The corresponding HDG method yields finite element
approximations $({\BQ}_h, U_h,\hat U_h)\in {\BV}_h^p\times
W_h^p\times M_h^{p,0}$ which satisfy
\begin{align}
\label{4P1}
({\BQ}_h,{\Br})_{\mathcal T_h}-(U_h, {{\rm div \, {\Br}}})_{\mathcal T_h}+\langle \hat U_h, {{\Br}\cdot {\Bn}}\rangle_{\partial \mathcal T_h}&=-\langle g, {{\Br}\cdot {\Bn}}\rangle_{\partial \Omega},\\
\label{4P2}
-({\BQ}_h,{\nabla w})_{\mathcal T_h}+\langle \hat {\BQ}_h\cdot {\Bn} , w\rangle_{\partial \mathcal T_h}&=(f, w)_{\mathcal T_h} ,\\
\label{4P3}
\langle \hat{\BQ}_h\cdot {\Bn},
\mu\rangle_{\partial \mathcal T_h }&=0,
\end{align}
for all ${\Br}\in {\BV}_h^p$, $w\in W_h^p$, and $\mu\in M_h^{p,0}$, where
$$ M_h^{p,0}:=\{ \mu\in M_h^p\ | \ \mu|_{\partial \Omega}=0\},$$
and
\begin{displaymath}
\hat {\BQ}_h\cdot {\Bn}=\left\{\begin{array}{ll}
{\BQ}_h\cdot {\Bn}+\tau_h (U_h-\hat U_h) & e\in \partial \mathcal T_h \backslash \partial \Omega; \\
{\BQ}_h\cdot {\Bn}+\tau_h (U_h-g)  & e\in \partial \mathcal T_h \cap \partial \Omega .
\end{array} \right.
\end{displaymath}
For fixed $p$, we choose $\tau_h=O(h^{-1})$ for the  Possion equation.

For any $m\in L^2(\partial T)$ and $f \in L^2(T)$, define $({\BQ}_m,{U}_m)$, $({\BQ}_f,{U}_f)\in V(T)\times W(T)$ as follows:
\begin{align}
\label{4localm1}
({\BQ}_m, \Bv)_T-(U_m, {\rm div}\ {\Bv})_T=-\langle m, {\Bv}\cdot {\Bn}\rangle_{\partial T},\\
\label{4localm2}
-({\BQ}_m,\nabla w)_T+\langle  \hat{{\BQ}}_m\cdot {\Bn},w\rangle_{\partial T}=0,
\end{align}
\begin{align}
\label{4localf1}
({\BQ}_f, \Bv)_T-(U_f, {\rm div}\ {\Bv})_T=0,\\
\label{4localf2}
-({\BQ}_f,\nabla w)_T+\langle  \hat{{\BQ}}_f\cdot {\Bn},w\rangle_{\partial T}=(f,w)_T,
\end{align}
for all $({\Bv},w)\in{\BV}(T)\times W(T)$,
where
$$\hat {{\BQ}}_m\cdot {\Bn}={\BQ}_m\cdot
{\Bn}+\tau_{h}(U_m-m),$$
$$\hat {{\BQ}}_f\cdot {\Bn}={\BQ}_f\cdot
{\Bn}+\tau_{h}U_f.$$
It is shown in Theorem 2.1 in \cite{Cockburn} that $\hat U_h\in M_h^{p,0}$ is the solution of the following equation
$$\hat a_h(\hat U_h,\mu)=\hat b_h(\mu),\quad \forall \  \mu\in M_h^{p,0},$$
where
\begin{align}
\label{bilinear-E}
\hat a_h(\lambda,\mu)&:=({\BQ}_{\lambda},{\BQ}_{\mu})_{\Ct_h}
+\tau_h\langle U_{\lambda}-\lambda, U_{\mu}-\mu\rangle_{\partial \Ct_h},\\
\hat b_h(\mu)&:=(f,U_\mu)_{\Ct_h}+\langle g, {\hat {\BQ}_\mu \cdot {\Bn}}\rangle_{\partial \Omega}.
\end{align}

Let the space $W_{h_l}^p$ and $M_{h_l}^{p,0}$ on $\Ct_l$ be denoted by $W_l$ and $M_l^0$ repectively,  while the local stabilization parameter is denoted by $\tau_l$, which is of order $O(h_l^{-1})$. Define the  average operator
$\tilde I_l^W:W_l\rightarrow W_l^c,0\leq l\leq L-1$ as follows:

\begin{displaymath}
\tilde I_l^W u(z_n)=\left\{\begin{array}{ll}
\frac{\sum_{T\subset \mathcal T_l(z_n)}u_T(z_n)}{|\mathcal T_l(z_n)|}& z_n \in \mathcal N_l\backslash   \mathcal N_l^0;\\
u(z_n)& z_n \in \mathcal N_l^0,
\end{array} \right.\\
\end{displaymath}
where $\mathcal T_l(z_n)$ is the collection of elements in $\Ct_l$ which contain $z_n$.
$|\mathcal T_l(z_n)|$ is the number of elements in
$\mathcal T_l(z_n)$. Note that the average operator $\tilde I_l^W$ coincides with $I_{os}$ in \cite{Erik}, we refer to \cite{Erik} for the properties of $\tilde I_l^W$.

\begin{lemma}
  \label{lem2.1}
For all $\mu\in M_l^0$, $T\in \Ct_l$, let $U_\mu$ be the solution of (\ref{4localm1}-\ref{4localm2}). Then
\begin{align}
\label{element}
|\tilde I_l^WU_\mu-I_l^W\mu|_{1,T}\lesssim h_l^{-1/2}\|U_\mu-\mu\|_{0,\partial \Ct_l(\Omega_l(T))},
\end{align}
where $\Omega_l(T):=\{T'\in \Ct_l, T'\cap T\not=\emptyset\}$, $\partial \Ct_l(\Omega_l(T)):=\bigcup_{T'\in \Omega_l(T)}\partial T'$. Furthermore,
summing up for all $T\in\Ct_l$, we have
\begin{align}
\label{imp}
|\tilde I_l^WU_\mu-I_l^W\mu|_{1,\Omega}\lesssim h_l^{-1/2}\|U_\mu-\mu\|_{0,\partial \Ct_l}.
\end{align}
\end{lemma}

\begin{proof}
For all $T_1\in\Ct_l$, if $\partial T_1\cap \partial \Omega=\emptyset$, suppose $\{a_i,i=1,\ldots, N_p\}$ are the degrees of freedom in $W(T_1)$.
It is obvious that
\begin{align}
|\tilde I_l^WU_\mu-I_l^W\mu|_{1,T_1}^2\lesssim \sum_{i=1}^{N_p}\big((\tilde I_l^WU_\mu-I_l^W\mu)(a_i)\big)^2.
\end{align}
According to the definition of $\tilde I_l^W$ and $I_l^W$,  we have (see Figure \ref{fig4.1})

\begin{figure}[!htbp]
\centering
   \includegraphics[width=2.4in]{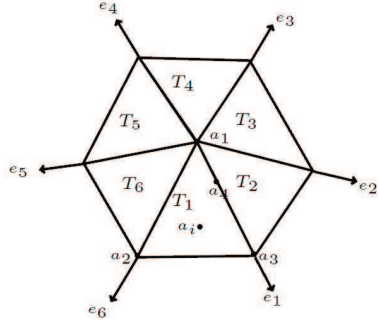}
   \caption{\small An illustration of the triangles containing $a_1$.}
   \label{fig4.1}
\end{figure}

\begin{align*}
(\tilde I_l^WU_\mu-I_l^W\mu)(a_1)&=\frac{U_{\mu}^{T_1}(a_1)+
\ldots +U_{\mu}^{T_{n_{a_1}}}(a_1)}{|\Ct_l(a_1)|}-\frac{{\mu}^{e_1}(a_1)+
\ldots +{\mu}^{e_{n_{a_1}}}(a_1)}{|\Ct_l(a_1)|}\\
&=\frac{(U_{\mu}^{T_1}(a_1)-{\mu}^{e_1}(a_1))+
\ldots +(U_{\mu}^{T_{n_{a_1}}}(a_1)-{\mu}^{e_{n_{a_1}}}(a_1))}{|\Ct_l(a_1)|},
\end{align*}
where $U_{\mu}^{T_i}$ and ${\mu}^{e_i},\ i=1,\ldots,{n_{a_1}}$, are the
values of $U_{\mu}$ in $T_i$ and ${\mu}$ in $e_i$ respectively. ${n_{a_1}}$ is the number of elements which share the vertex $a_1$, specially for the case in Figure \ref{fig4.1}, ${n_{a_1}}=6$.

Since $\Ct_l$ is a sharp regular mesh,
\begin{align*}
\big((\tilde I_l^WU_\mu-I_l^W\mu)(a_1)\big)^2&\lesssim \big((U_{\mu}^{T_1}(a_1)-{\mu}^{e_1}(a_1))\big)^2+
\ldots +\big((U_{\mu}^{T_{n_{a_1}}}(a_1)-{\mu}^{e_{n_{a_1}}}(a_1))\big)^2\\
&\lesssim h_l^{-1}\|U_\mu-\mu\|_{0,\partial \Ct_l(\Omega_l(a_1))}^2,
\end{align*}
where $\Omega_l(a_1):=\{T'\in \Ct_l, T'\ {\rm {contains }}\ a_1\}$, $\partial \Ct_l(\Omega_l(a_1)):=\bigcup_{T'\in \Omega_l(a_1)}\partial T'$.
Similarly we can get
\begin{align*}
\big((\tilde I_l^WU_\mu-I_l^W\mu)(a_2)\big)^2
\lesssim h_l^{-1}\|U_\mu-\mu\|_{0,\partial \Ct_l(\Omega_l(a_2))}^2,
\end{align*}
and
\begin{align*}
\big((\tilde I_l^WU_\mu-I_l^W\mu)(a_3)\big)^2
\lesssim h_l^{-1}\|U_\mu-\mu\|_{0,\partial \Ct_l(\Omega_l(a_3))}^2.
\end{align*}
Then for the case $p=1$, (\ref{element}) is proved.
Suppose $a_4$ is the degree of freedom of $W(T_1)$ in the edge $[a_1, a_3]$, by the definition of $\tilde I_l^W$ and $I_l^W$, we have
\begin{align*}
(\tilde I_l^WU_\mu-I_l^W\mu)(a_4)=
\frac{(U_{\mu}^{T_1}(a_4)-{\mu}(a_4)) +(U_{\mu}^{T_2}(a_4)-{\mu}(a_4))}{2},
\end{align*}
which means
\begin{align*}
\big((\tilde I_l^WU_\mu-I_l^W\mu)(a_4)\big)^2
\lesssim h_l^{-1}\|U_\mu-\mu\|_{0,\partial \Ct_l(\Omega_l(a_4))}^2.
\end{align*}
Similarly we can obtain the estimates for the degrees of freedom of $W(T_1)$ in the edges $[a_1, a_2]$ and $[a_2, a_3]$.
Then for the case $p=2$, (\ref{element}) is proved.
Suppose $a_i$ is the degree of freedom of $W(T_1)$ in the interior of $T_1$, since
$\tilde I_l^WU_\mu(a_i)=U_\mu(a_i)$, according to the definition of $I_l^W\mu$, we can get
$$\tilde I_l^WU_\mu(a_i)=I_l^W\mu(a_i).$$
Hence for the case $p\geq 3$, (\ref{element}) is proved.

If $\partial T_1\cap \partial \Omega\not=\emptyset$, due to the fact that $\mu\in M_l^0$, (\ref{element}) can be derived similarly.
\end{proof}

\begin{remark}
From the proof of
Lemma \ref{lem2.1}, we may see  why we should define the  degree of freedom in the interior of $T_1$ by $U_\lambda$ in  the  case $p\geq 3$. Actually, this definition may ensure (\ref{imp}), which is important in our analysis for the stability estimate of $I_l$.
\end{remark}

\begin{lemma}
  \label{lem2.2}
For all $T\in\Ct_l$, $m\in L^2(\partial T)$, let $({\BQ}_m,U_m)$ be the solution of the local problem (\ref{4localm1}-\ref{4localm2}). Then
\begin{align}
\|{\BQ}_m+\nabla U_m\|_{0,T}\lesssim h_l^{-1/2}\|U_m-m\|_{0,\partial T}.
\end{align}
\end{lemma}

\begin{proof}
Applying the Green's formula and using (\ref{4localm1}), we have
$$({\BQ}_m+\nabla U_m,{\Bv})_T=\langle U_m-m, {\Bv}\cdot {\Bn}\rangle_{\partial T}, \quad \forall {\Bv}\in V(T).$$
Taking $v={\BQ}_m+\nabla U_m$, and using Lemma 3.2 in  \cite{Wuhp}, we get
\begin{align*}
\|{\BQ}_m+\nabla U_m\|_{0,T}^2&\leq \|U_m-m\|_{0,\partial T}\|{\BQ}_m+\nabla U_m\|_{0,\partial T}\\
&\lesssim h_l^{-1/2}\|U_m-m\|_{0,\partial T}\|{\BQ}_m+\nabla U_m\|_{0,T},
\end{align*}
Eliminating $\|{\BQ}_m+\nabla U_m\|_{0,T}$ from both side of the
above inequality concludes the proof.
\end{proof}

\begin{lemma}
  \label{allp}
For all $\mu\in M_l^0$, we have
  $$|I_l^W\mu|_{1,\Omega}^2\lesssim \hat a_l(\mu,\mu).$$

\end{lemma}
\begin{proof}
By the triangle inequality and Lemma \ref{lem2.1}, we obtain
\begin{align*}
|I_l^W\mu|_{1,\Omega}&\leq |I_l^W\mu-\tilde I_l^WU_\mu|_{1,\Omega}+|\tilde I_l^WU_\mu|_{1,\Omega},\\
&\lesssim h_l^{-1/2}\|U_\mu-\mu\|_{0,\partial \Ct_l}+|\tilde I_l^WU_\mu|_{1,\Omega}.
\end{align*}
The property of the operator $\tilde I_l^W$ (see Remark 3.2 of \cite{Erik}) implies
\begin{align*}
|\tilde I_l^WU_\mu|_{1,\Omega}&\lesssim |U_\mu|_{1,\Omega_l}+h_l^{-1/2}\|\llbracket U_\mu\rrbracket\|_{0,\mathcal E_l^0}\\
&\lesssim |U_\mu|_{1,\Omega_l}+h_l^{-1/2}\|U_\mu-\mu\|_{0,\partial \Ct_l},
\end{align*}
where $|v|_{1,\Omega_l}^2:=\sum_{T\in \Ct_l}|v|_{1,T}^2, \ \forall v\in W_l$.
Then using Lemma \ref{lem2.2}, we have
\begin{align*}
|U_\mu|_{1,\Omega_l}\lesssim \|{\BQ}_\mu\|_{0,\Omega}+h_l^{-1/2}\|U_\mu-\mu\|_{0,\partial \Ct_l}.
\end{align*}
Hence
\begin{align*}
|I_l^W\mu|_{1,\Omega}^2\lesssim \|{\BQ}_\mu\|_{0,\Omega}^2+h_l^{-1}\|U_\mu-\mu\|_{0,\partial \Ct_l}^2\lesssim \hat a_l(\mu,\mu).
\end{align*}
\end{proof}

Next, we prove the stability estimate of the intergrid transfer operator  for the case $p=1$.

\begin{lemma}
  \label{lem2.3}
For the case $p=1$,  let $({\BQ}_{I_l\mu},U_{I_l\mu})\in {\BV}(K)\times W(K)$ be the solution of the local problem (\ref{4localm1}-\ref{4localm2}) for any $K\in \Ct_L$, $I_l\mu\in L^2(\partial K)$, where $\mu\in M_l$. Then
\begin{align}
{\BQ}_{I_l\mu}=-\nabla I_l^W\mu,\qquad U_{I_l\mu}=I_l^W\mu.
\end{align}
\end{lemma}
\begin{proof}
The Green's formula and (\ref{trans}) implies
\begin{align*}
(-\nabla I_l^W\mu,{\Bv})_K-(I_l^W\mu,{\rm div}\ {\Bv})_K=
-\langle I_l^W\mu,{\Bv}\cdot {\Bn}\rangle_{\partial K}=
-\langle I_l\mu,{\Bv}\cdot {\Bn}\rangle_{\partial K},
\end{align*}
for all ${\Bv}\in {\BV}(K)$.
Since $I_l^W\mu$ is piecewise linear, then we have
\begin{align}
\label{linear} {\rm div} \ (-\nabla I_l^W\mu)=-\Delta I_l^W\mu=0 ,
\end{align}
which together with (\ref{trans}) yields
\begin{align*}
({\rm div} \ (-\nabla I_l^W\mu),w)_K+\tau_L\langle I_l^W\mu-I_l\mu, w\rangle_{\partial K}=0,\quad \forall w \in W(K).
\end{align*}

Because the local problem (\ref{4localm1}-\ref{4localm2}) is uniquely solvable,  we  derive
\begin{align*}
{\BQ}_{I_l\mu}=-\nabla I_l^W\mu,\qquad U_{I_l\mu}=I_l^W\mu.
\end{align*}
\end{proof}
\begin{theorem}
When $p=1$, the intergrid transfer operator $I_l: M_l\rightarrow M_L$ satisfies
\begin{align*}
\hat a_L(I_l\mu,I_l\mu)\lesssim \hat a_l(\mu,\mu),\quad \forall \mu\in M_l^0.
\end{align*}

\end{theorem}
\begin{proof}
The definition of the bilinear form $\hat a_L(\cdot,\cdot)$, Lemma \ref{lem2.3} and (\ref{trans}) yield
\begin{align*}
\hat a_L(I_l\mu,I_l\mu)=\sum_{K\in \Ct_L}\|{\BQ}_{I_l\mu}\|_{0,K}^2=|I_l^W\mu|_{1,\Omega}^2,
\end{align*}
which, together with Lemma \ref{allp}, implies the conclusion.
\end{proof}

For the case $p\geq 2$, since for all $v\in P_p(T)$, $\Delta v$ no
longer equals to 0, which means Lemma \ref{lem2.3} doesn't hold
any more. We will prove the stability estimate for $I_l$ through the
estimation of terms $\|{\BQ}_{I_l\mu}+\nabla I_l^W\mu\|_{0,\Omega}$ and
$\|U_{I_l\mu}-I_l^W\mu\|_{0,\partial \Ct_l}$.

\begin{lemma}
\label{lem2.5}
For all $T\in\Ct_l$, $f\in L^2(T)$, let $({\BQ}_f, U_f)\in{\BV}(T)\times W(T)$ be the solution of the local problem (\ref{4localf1}-\ref{4localf2}). Then
\begin{align}
\|{\BQ}_f\|_{0,T}&\lesssim h_l\|f\|_{0,T},\\
\|U_f\|_{0,T}\lesssim h_l^2\|f\|_{0,T},&\quad \|U_f\|_{0,\partial T}\lesssim h_l^{3/2}\|f\|_{0,T}.
\end{align}
\end{lemma}

\begin{proof}
Using the Green's formula,  we know there exists a $({\BQ}_f, U_f)\in{\BV}(T)\times W(T)$ such that
\begin{align}
\label{4localf1G}
({\BQ}_f, \Bv)_T-(U_f, {\rm div}\ {\Bv})_T=0,\\
\label{4localf2G}
({\rm div}\ {\BQ}_f, w)_T+\tau_l\langle  U_f,w\rangle_{\partial T}=(f,w)_T,
\end{align}
for all $({\Bv},w)\in{\BV}(T)\times W(T)$.

Next we prove that for all $w\in P_p(T)$, there holds
\begin{align}
\label{equiv}
\|w\|_{0,T}\lesssim h_l\frac{{\rm sup}_{{\Bv}\in V(T)}\ |\int_T w {\rm div}\ {\Bv}|}{\|{\Bv}\|_{0,T}}
+h_l^{1/2}\|w\|_{0,\partial T}.
\end{align}
Supposing $\hat T$ is the standard triangle element,  and $F_T:\hat T\rightarrow T$  is a linear  map which is defined by $x=F_T(\hat x):=B_T\hat x+b$.
A scalar function $w$ on $T$ is transformed to a scalar function $\hat w$ on $\hat T$ by $\hat w:=w\circ F_T(\hat x)$, while  for the vector function ${\Bv}$, we transform ${\Bv}$ to $\hat {\Bv}$ via  $\hat {\Bv}:={\rm det}(B_T)B_T^{-1} {\Bv}\circ F_T(\hat x)$. We note that this is a divergence conserving transformation (see Lemma 3.59 in \cite{peter}), i.e.
$${\rm div}\ \hat {\Bv}={\rm det}(B_T){\rm div}\ {\Bv}.$$
Define
$$\|\hat w\|:=\frac{{\rm sup}_{\hat{\Bv}\in (P_p(\hat T))^2} |\int_{\hat T} \hat w {\rm div}\ \hat{\Bv}|}{\|\hat{\Bv}\|_{0,\hat T}}
+\|\hat w\|_{0,\partial \hat T}.$$ Now we prove that $\|\cdot\|$ is
a norm in the space $P_p(\hat T)$. Apparently we have
\begin{align*}
 \|c\hat w\|=|c| \ \|\hat w\|,\\
 \|\hat w+\hat u\|\leq \|\hat w\|+\|\hat u\|,
\end{align*}
for all $\hat w,\ \hat u\in P_p(\hat T)$ and $c\in \mathbb R$. Hence
we only need to verify that if $\|\hat w\|=0$, then $w=0$.

If $\|\hat w\|=0$, by the definition of $\|\cdot\|$, we have
$$\int_{\hat T}\hat w \hat u_{p-1}=0,\quad \forall \hat u_{p-1}\in P_{p-1}(\hat T).$$
By Lemma A.1 in \cite{Cockburn2}, we know
\begin{align*}
\|\hat w\|_{0,\hat T}\lesssim \|\hat w\|_{0,\hat F},\quad \forall
\hat F\in \partial \hat T.
\end{align*}
which implies $\hat w=0$. The scaling argument and the equivalence of
the norms  $\|\cdot\|_{0,\hat T}$ and  $\|\cdot\|$ in the finite
dimensional space $P_p(\hat T)$ yield
\begin{align}
\nn
h_l^{-1}\|w\|_{0,T}&\lesssim \|\hat w\|_{0,\hat T}\lesssim \frac{{\rm sup}_{\hat{\Bv}\in (P_p(\hat T))^2} |\int_{\hat T} \hat w {\rm div}\ \hat{\Bv}|}{\|\hat{\Bv}\|_{0,\hat T}}+\|\hat w\|_{0,\partial \hat T}\\
\nn
&\lesssim \frac{{\rm sup}_{{\Bv}\in V(T)}\ |\int_T w {\rm div}\ {\Bv}|}{\|{\Bv}\|_{0,T}}
+h_l^{-1/2}\|w\|_{0,\partial T},
\end{align}
which results in (\ref{equiv}).
Taking ${\Bv}={\BQ}_f$ in (\ref{4localf1G}) and  $w=U_f$ in (\ref{4localf2G}), we can deduce
\begin{align}
\label{Qf}
\|{\BQ}_f\|_{0,T}^2+\tau_l\|U_f\|_{0,\partial T}^2=(f,U_f)_T.
\end{align}
By (\ref{equiv}) and (\ref{4localf1G}), we get
\begin{align}
\nn
\|U_f\|_{0,T}&\lesssim h_l\frac{{\rm sup}_{{\Bv}\in V(T)}\ |\int_T U_f {\rm div}\ {\Bv}|}{\|{\Bv}\|_{0,T}}
+h_l^{1/2}\|U_f\|_{0,\partial T}\\
\nn
&\lesssim h_l\frac{{\rm sup}_{{\Bv}\in V(T)}\ |\int_T {\BQ}_f  {\Bv}|}{\|{\Bv}\|_{0,T}}
+h_l^{1/2}\|U_f\|_{0,\partial T}\\
\label{Uf}
&\lesssim h_l \|{\BQ}_f\|_{0,T}+h_l^{1/2}\|U_f\|_{0,\partial T}.
\end{align}
Taking (\ref{Uf}) to (\ref{Qf}), and utilizing the Young's inequality gives
\begin{align*}
\|{\BQ}_f\|_{0,T}^2+\tau_l\|U_f\|_{0,\partial T}^2&\lesssim \|f\|_{0,T}(h_l \|{\BQ}_f\|_{0,T}+h_l^{1/2}\|U_f\|_{0,\partial T})\\
&\leq C_{\delta}h_l^2\|f\|_{0,T}^2+\delta \|{\BQ}_f\|_{0,T}^2+\delta h_l^{-1}\|U_f\|_{0,\partial T}^2.
\end{align*}
Choosing $\delta<{\rm min}\ \{ \frac{1}{2}, \frac{1}{2}\tau_l h_l \}$, we obtain
\begin{align*}
\|{\BQ}_f\|_{0,T}^2+\tau_l\|U_f\|_{0,\partial T}^2\lesssim h_l^2\|f\|_{0,T}^2.
\end{align*}
Hence
\begin{align*}
\|{\BQ}_f\|_{0,T}\lesssim h_l\|f\|_{0,T}, \quad \|U_f\|_{0,\partial T}\lesssim h_l^{3/2}\|f\|_{0,T}.
\end{align*}
By (\ref{Uf}), we  derive
\begin{align*}
\|U_f\|_{0,T}\lesssim h_l^2\|f\|_{0,T}.
\end{align*}
\end{proof}

Now we are ready to prove the stability  of the intergrid transfer operator for the case $p\geq 2$.

\begin{theorem}
When $p\geq 2$, the intergrid transfer operator $I_l: M_l\rightarrow M_L$ satisfies
\begin{align*}
\hat a_L(I_l\mu,I_l\mu)\lesssim \hat a_l(\mu,\mu),\quad \forall \mu\in M_l^0.
\end{align*}
\end{theorem}

\begin{proof}
Let $({\BQ}_{I_l\mu},U_{I_l\mu})\in {\BV}(K)\times W(K)$ be the solution of the local problem (\ref{4localm1}-\ref{4localm2}) for any $K\in \Ct_L$, $I_l\mu\in L^2(\partial K)$. Then using the Green's formula, we know that there exists a
$({\BQ}_{I_l\mu}, U_{I_l\mu})\in{\BV}(K)\times W(K)$ such that
\begin{align}
\label{4localm1G}
({\BQ}_{I_l\mu}, \Bv)_K-(U_{I_l\mu}, {\rm div}\ {\Bv})_T=-\langle I_l\mu, {\Bv}\cdot {\Bn}\rangle_{\partial K},\\
\label{4localm2G}
({\rm div}\ {\BQ}_{I_l\mu}, w)_K+\tau_L\langle  U_{I_l\mu}-I_l\mu,w\rangle_{\partial K}=0,
\end{align}
for all $({\Bv},w)\in{\BV}(K)\times W(K)$.
The Green's formula and (\ref{trans}) imply
\begin{align}
\label{locala1}
(-\nabla I_l^W\mu,{\Bv})_K-(I_l^W\mu,{\rm div}\ {\Bv})_K=
-\langle I_l^W\mu,{\Bv}\cdot {\Bn}\rangle_{\partial K}=
-\langle I_l\mu,{\Bv}\cdot {\Bn}\rangle_{\partial K},\\
\label{locala2} ({\rm div} \ (-\nabla I_l^W\mu),w)_K+\tau_L\langle
I_l^W\mu-I_l\mu, w\rangle_{\partial K}=(-\Delta I_l^W\mu,w)_K,
\end{align}
for all $({\Bv},w)\in{\BV}(K)\times W(K)$.
Denote $e_{{\BQ}}={\BQ}_{I_l\mu}+\nabla I_l^W\mu$ and $e_U=U_{I_l\mu}-I_l^W\mu$. Then
$(e_{{\BQ}}, e_U)\in{\BV}(K)\times W(K)$ satisfy
\begin{align}
(e_{{\BQ}}, \Bv)_K-(e_U, {\rm div}\ {\Bv})_K=0, \\
({\rm div}\ e_{{\BQ}}, w)_K+\tau_L\langle  e_U,w\rangle_{\partial
K}=(\Delta I_l^W\mu,w)_K,
\end{align}
for all $({\Bv},w)\in{\BV}(K)\times W(K)$.
By Lemma \ref{lem2.5}, we  get
\begin{align*}
\|e_{{\BQ}}\|_{0,K}\lesssim h_L \|\Delta I_l^W\mu\|_{0,K},\\
\|e_{U}\|_{0,K}\lesssim h_L^2 \|\Delta I_l^W\mu\|_{0,K},\\
\|e_{U}\|_{0,\partial K}\lesssim h_L^{3/2} \|\Delta
I_l^W\mu\|_{0,K}.
\end{align*}
Summing up for all $K\in\Ct_L$ and utilizing the inverse inequality, we have
\begin{align}
\|e_{{\BQ}}\|_{0,\Omega}\lesssim \frac{h_L}{h_l} |I_l^W\mu|_{1,\Omega}\lesssim |I_l^W\mu|_{1,\Omega},\\
\|e_{U}\|_{0,\Omega}\lesssim  \frac{h_L^2}{h_l} |I_l^W\mu|_{1,\Omega}\lesssim h_L|I_l^W\mu|_{1,\Omega},\\
\|e_{U}\|_{0,\partial \Ct_L}\lesssim  \frac{h_L^{3/2}}{h_l} |I_l^W\mu|_{1,\Omega}\lesssim h_L^{1/2}|I_l^W\mu|_{1,\Omega}.
\end{align}
Then we can deduce
$$\|{\BQ}_{I_l\mu}\|_{0,\Omega}\lesssim |I_l^W\mu|_{1,\Omega}+\|e_{{\BQ}}\|_{0,\Omega}\lesssim |I_l^W\mu|_{1,\Omega}$$
and
$$\|U_{I_l\mu}-I_l\mu\|_{0,\partial \Ct_L}=\|e_U\|_{0,\partial \Ct_L}.$$
Hence
\begin{align*}
\hat a_L(I_l\mu,I_l\mu)&=\|{\BQ}_{I_l\mu}\|_{0,\Omega}^2+\tau_L\|U_{I_l\mu}-I_l\mu\|_{0,\partial \Ct_L}^2\\
&\lesssim |I_l^W\mu|_{1,\Omega}^2+\tau_Lh_L|I_l^W\mu|_{1,\Omega}^2\\
&\lesssim |I_l^W\mu|_{1,\Omega}^2,
\end{align*}
which, together with Lemma \ref{allp}, complete the proof.
\end{proof}

\begin{remark}
In this paper, the proof of the stability estimate of the intergrid transfer operator is specified in meshes consisting of triangles, but we should mention that it can be extended to
the meshes constituted with rectangles, tetrahedra or hexahedra.
\end{remark}

\section{Local Fourier analysis (LFA)}\label{LFA}

In this section, LFA will be used to give a quantitative insight of
the convergence of Algorithm \ref{alg} in 1D case. For simplicity, in this section we focus on the analysis for the HDG discretization based on linear polynomial (P1) approximation (HDG-P1). We mainly consider the
analysis of two level method of Algorithm \ref{alg}. The LFA of
three level method is also mentioned. The analysis imply the
efficiency of Algorithm \ref{alg}. We adopt the notations and
philosophy in \cite{WJ04}.


There are some necessary simplifications in the framework of LFA: the boundary conditions are neglected and the
problem is considered on regular indefinite grids $\Cg_h = \{ x:
x=x_j=jh, j \in \mathbb{Z} \}$. It seems to be very restrictive and very unrealistic since the Robin boundary condition (\ref{HE}) and other
absorbing boundary conditions are often applied in realistic Helmholtz problem, the neglect of boundary  conditions does usually not affect the validity of LFA (cf. \cite{WJ04}).

For a fixed point $x\in \Cg_h$ and any infinite grid function $u_h$, we can define an operator $\BL_h$ on the space of infinite grid functions by
\begin{align}
\BL_h u_h(x) = \sum_{j \in J} l_j u_h(x+jh), \nn
\end{align}
with stencil coefficients $l_j$ and a certain finite subset $J\subset \mathbb{Z}$. $\BL_h = [l_j]_h, j \in J$ is a stencil representation of $\BL_h$. This formulation is particularly convenient in the context of LFA. It can be easily seen that the eigenfunctions of $\BL_h$ are given by  $
\varphi_h(\theta,x) = e^{{\bf i} \theta x/h}$ with $x \in \Cg_h$ and
 $\theta \in \mathbb{R}$. In fact, the frequency
$\theta$ can be restricted to the interval $(-\pi,\pi]$ as a fact
that $\varphi_h(\theta + 2\pi,x) = \varphi_h(\theta,x)$. These eigenfunctions are called Fourier components associated with a Fourier frequency $\theta$. The corresponding eigenvalues of $\BL_h$ which are called Fourier symbols read as
$
\widetilde{\BL}_h(\theta) = \sum_{j\in J} l_j e^{{\bf i} \theta j}$, and satisfy the following equality
\begin{align}
\BL_h \varphi_h(\theta,x) = \widetilde{\BL}_h(\theta)
\varphi_h(\theta,x), \qquad x\in \Cg_h , \theta \in (-\pi,\pi].
\label{Lh-stencil}
\end{align}
 Given a
so-called low frequency $\theta^0\in \Theta_{\rm
low}:=(-\pi/2,\pi/2]$, its complementary frequency $\theta^1$ is
defined as
\begin{align}
\theta^1 = \theta^0-{\rm sign}(\theta^0)\pi. \label{theta1}
\end{align}
It is appropriate to divide the Fourier space into the following two
dimensional subspace
\begin{align}
E_{2h}^{\theta^0} := {\rm span} \{
\varphi_h(\theta^0,x),\varphi_h(\theta^1,x)  \}, \label{Etheta}
\end{align}
where the Fourier components $\varphi_h(\theta^0,x) $ and
$\varphi_h(\theta^1,x) $ are called $2h$-harmonics. The definition
of the $2h$-harmonics is motivated by the fact that each low
frequency $\theta^0\in \Theta_{\rm low}$ is coupled with $\theta^1$
in the transition from $\Cg_h$ to $\Cg_{2h}$. Indeed they coincide
with each other on the coarse grid. Interpreting the Fourier
components as coarse grid functions yields
\[
\varphi_h(\theta^0,x) = \varphi_{2h}(2\theta^0,x) =
\varphi_{2h}(2\theta^1,x) = \varphi_h(\theta^1,x), \qquad \theta^0
\in \Theta_{\rm low}, \ x\in \Cg_{2h}.
\]
A crucial observation is that the space $E_{2h}^{\theta^0}$ is
invariant under both smoothing operators and correction schemes for
general cases by two level method. The invariance property holds for
many well-known smoothing methods (cf. \cite{WJ04}), such as Jacobi
relaxation, lexicographical Gauss-Seidel relaxation, et al.

The main goal of LFA is to estimate the spectral radius or certain
norms of the $k$-level operator. Let $\BM_h$ be a discrete two level
operator. In the following we will show that a block-diagonal
representation for $\BM_h$ consists of $2\times 2$ blocks
$\widetilde{\BM}_h(\theta)$ (cf. \cite{WJ04}), which denotes the
representation of $\BM_h$ on $E_{2h}^{\theta^0}$. Then the
convergence factor of $\BM_h$ by the LFA is defined as follows:
\[
\rho(\BM_h) := {\rm sup} \{ \rho(\widetilde{\BM}_h(\theta)): \ \theta
\in \Theta_{\rm low} \},
\]
where $ \rho(\widetilde{\BM}_h(\theta))$ is the spectral radius of
the matrix $\widetilde{\BM}_h(\theta)$. The generalizations to
$k$-level analysis are shown in \cite{WJ04}.

\subsection{One dimensional Fourier symbols}

In this subsection, we give the Fourier symbols of different
operators in multilevel method for the HDG-P1 discretization for one
dimensional Helmholtz equation. Since the boundary condition is
neglected in the LFA, the stencil presentation of discretization
operator $\BA_h$ from (\ref{Ah}) can be derived as
\begin{align}
\BA_h = [ s_1  \ \ s_0 \ \ s_1
]_h,\nn
\end{align}
where
\begin{align*}
 s_1=\frac{1}{t}\Big(\frac{\sigma_2}{t \sigma_1}-\frac{\sigma_2(3t^2{\bf i}+18)}{(18t-t^3)\sigma_1}\Big)-\frac{-2t^4{\bf i}+12t^2{\bf i}+72+136{\bf i}}{t^5-24t^3+148t}
 -\frac{6t^2-72+12{\bf i}}{t^5-24t^3+148t},
 \end{align*}
 and
 \begin{align*}
 s_0&=-\frac{1}{t}\Big(\frac{\sigma_3}{ \sigma_4}-\frac{-4t^7{\bf i}+t^5(20+84{\bf i})+t^3(-432-480{\bf i})+t(2736+432{\bf i})}{t(t^8-24t^6+184t^4-864t^2+5328)}+1\Big)\\
 & \quad -\frac{-4t^4{\bf i}+60t^2{\bf i}+72-160{\bf i}}{t^5-24t^3+148t}
 -\frac{\sigma_3}{t\sigma_4},
 \end{align*}
 here $t=kh$, and
 \begin{align*}
 &\sigma_1=t^4{\bf i}+t^2(8-12{\bf i})-72-12{\bf i}, \quad
 \sigma_2=36t-2t^3,\\
  &\sigma_3=6t^2-72+12{\bf i},\quad
  \sigma_4=t^4-24t^2+148.
\end{align*}
Combining the above expression and (\ref{Lh-stencil}) yields the
Fourier symbol of $\BA_h$ as
\begin{align}
\widetilde{\BA}_h(\theta) =2s_1 \cos{\theta} +s_0 . \label{Ah-repre}
\end{align}
For simplicity, we use standard weighted Jacobi ($\omega$-JAC) and
lexicographical Gauss-Seidel (GS-LEX) relaxations as the smoothers
in the LFA. It is easy to derive the weighted Jacobi relaxation
matrix as $\BS^J_h = \BI_h - \omega \BD^{-1}_h \BA_h$, where
$\BI_h$ is indentity matrix, $\BD_h$ consists of the diagonal of
$\BA_h$ and $\omega$ is a weighted parameter. Due to the fact that
$\BD_h = s_0\BI_h$, one can easily deduce the Fourier
symbol of weighted Jacobi relaxation as follows:
\begin{align}
\widetilde{\BS}^J_h (\theta) = 1-\frac{2\omega}{s_0}( s_1 \cos{\theta} +\frac{s_0}{2} ). \label{SJh-repre}
\end{align}
The GS-LEX relaxation matrix is $\BS^{GS}_h = (\BD_h -
\BL_h)^{-1}\BU_h$, where $-\BL_h$ is the strictly lower triangular
part of $\BA_h$ and $-\BU_h$ is the strictly upper triangular part
of $\BA_h$. The Fourier symbol of $\BS^{GS}_h$ can also be
directly derived that
\begin{align}
\widetilde{\BS}^{GS}_h (\theta) = - \frac{ s_1 e^{{\bf i}\theta}}{ s_1 e^{{\bf -i}\theta}+s_0}.\label{GSJh-repre}
\end{align}

Note that for the restriction matrix $\BI^{2h}_h=[r_j]^{2h}_h$ and
$x\in \Cg_{2h}$, there holds
\[
(\BI^{2h}_h\varphi_h(\theta^\alpha,\cdot))(x) = \sum_{j\in J}r_j
e^{{\bf i} j \theta^\alpha}\varphi_h(\theta^\alpha,x) = \sum_{j\in
J}r_j e^{{\bf i} j \theta^\alpha}\varphi_{2h}(2\theta^0,x),\qquad
\alpha = 0,1.
\]
By an analogous stencil argument, the stencil presentation of full
weighting restriction matrix for the HDG-P1 discretization system in
one dimensional case is derived to be $\BI^{2h}_h = [1/4, 1/2,
1/4]_h^{2h}$. Thus, the Fourier symbol of $\BI^{2h}_h$ can be
deduced as
\[
\widetilde{\BI}^{2h}_h(\theta)  = \frac{1}{2} (1+\cos{\theta}).
\]
For the linear prolongation matrix $\BI^{h}_{2h}$ which is defined as
\begin{align*}
(\BI^{h}_{2h} \varphi_{2h} (2\theta^0,\cdot))(x) =
\varphi_{2h}(2\theta^0,x) = \varphi_h(\theta^0,x),\qquad x\in
\Cg_{2h}, \theta^0 \in \Theta_{\rm low},\\
(\BI^{h}_{2h} \varphi_{2h} (2\theta^0,\cdot))(x) = \frac{1}{2} (
\varphi_h(\theta^0,x-h) + \varphi_h(\theta^0,x+h) ), \qquad x \in
\Cg_h\setminus \Cg_{2h}, \theta^0 \in \Theta_{\rm low},
\end{align*}
one can also
obtain its Fourier symbol as follows (cf. \cite{WJ04}):
\[
\widetilde{\BI}^{h}_{2h}(\theta)  = \frac{1}{2} (1+\cos{\theta}).
\]

\subsection{Smoothing analysis}
Since every two dimensional subspace of $2h$-harmonics
$E^{\theta^0}_{2h}$ with $\theta^0 \in \Theta_{\rm low}$ is left
invariant under the $\omega$-JAC and GS-LEX relaxations, then the
Fourier representation of smoother $\BS_h=\BS^J_h$ or $\BS^{GS}_h$
with respect to $E_{2h}^{\theta^0}$ can be written as
\begin{equation}
\left[
  \begin{array}{cc}
    \widetilde{\BS}_h(\theta^0)  & 0  \\
    0 & \widetilde{\BS}_h(\theta^1)\\
  \end{array}
\right],
\end{equation}
where $\widetilde{\BS}_h(\theta)$ is the smoother symbol derived in
(\ref{SJh-repre}) and (\ref{GSJh-repre}). The spectral radius of the
smoother operator can be easily calculated since the above matrix is
diagonal.

We concern on the LFA for HDG-P1 method. The left graph of Figure \ref{fig-jacobi-1} shows the
Fourier symbols $\widetilde{\BS}^J_h(\theta)$ for $\omega$-JAC
smoother with $\omega=0.6$. We find that
$\widetilde{\BS}^J_h(\theta)\geq 1$ always occur at the low
frequencies, and small $t$ leads to a better relaxation. Similar
phenomenon is also observed for GS-LEX smoother in the right graph
of Figure \ref{fig-jacobi-1}. Thus, for fixed wave number $\kappa$,
both $\omega$-JAC and GS-LEX relaxations can be used as smoother on
fine grids, but on coarse grids they may amplify the error.
\begin{figure}[!h]
\centering
    \includegraphics[width=2.4in]{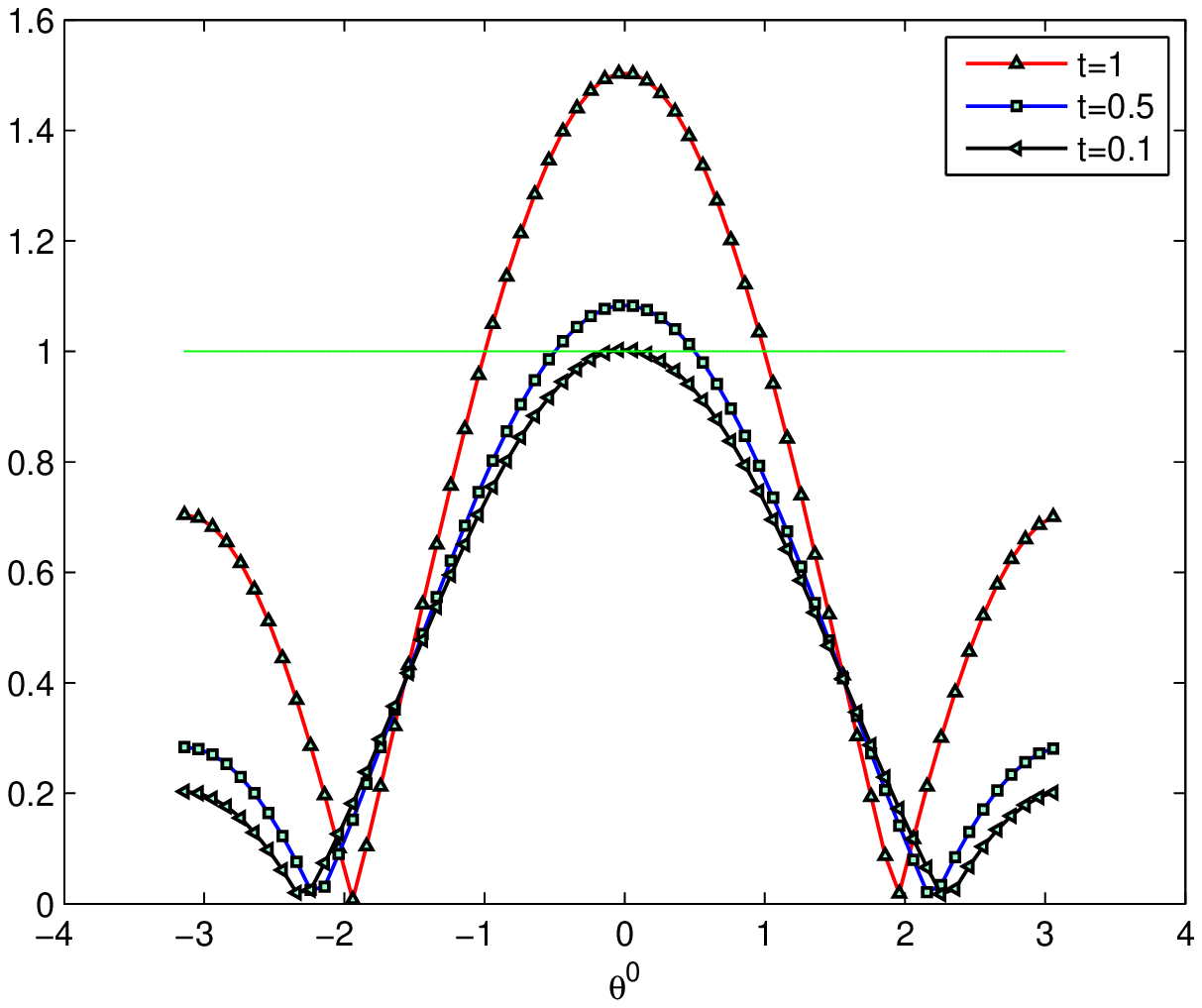}
    \includegraphics[width=2.4in]{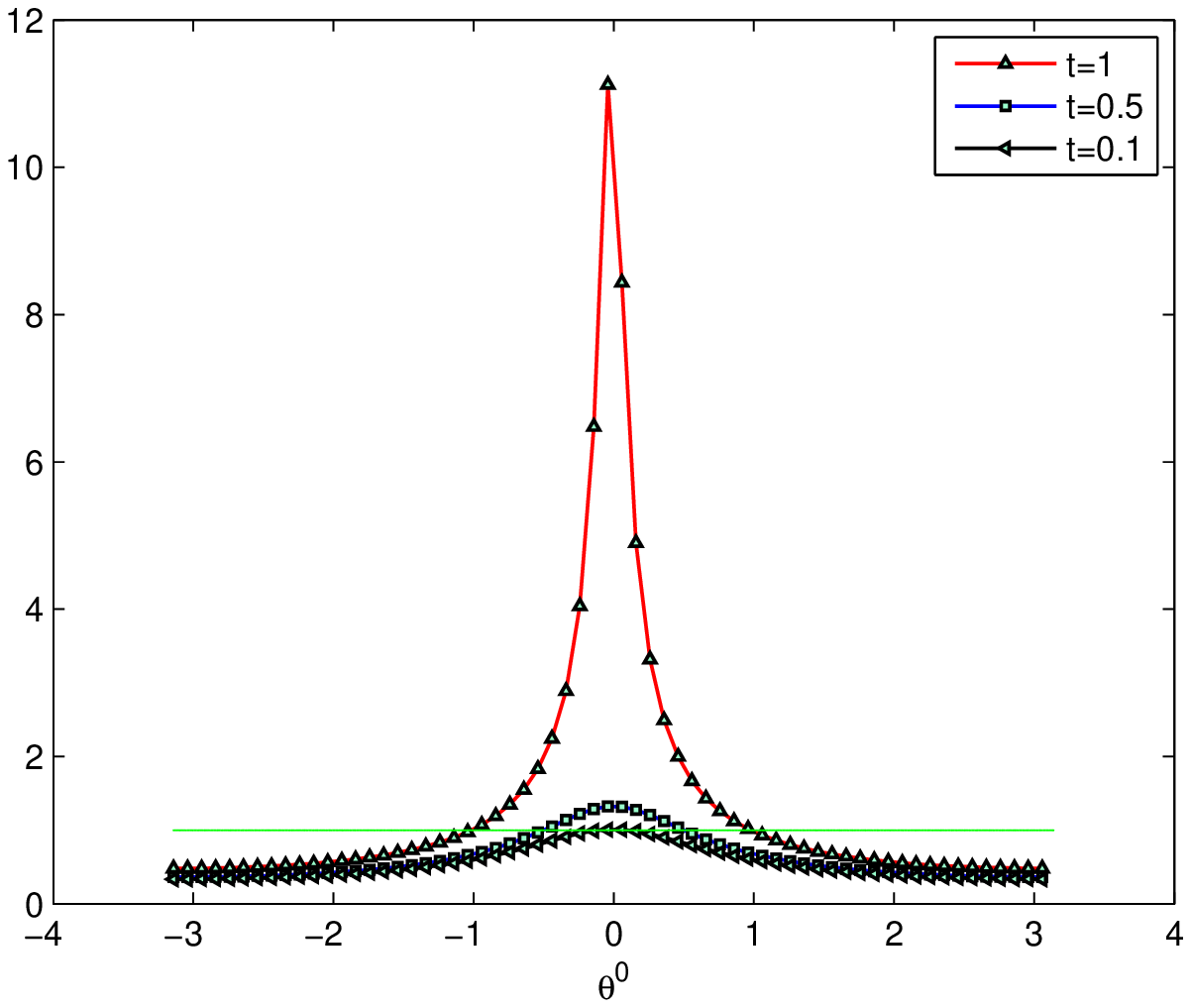}
    \caption{\small $|\widetilde{\BS}^J_h(\theta)|$ with $\omega=0.6$ (left) and $|\widetilde{\BS}^{GS}_h(\theta)|$ (right) over $(-\pi,\pi]$ for $t=0.1,0.5,1$.  }\label{fig-jacobi-1}
\end{figure}
Motivated by the idea in \cite{EEO01}, we use GMRES smoothing on
coarse grids. Unfortunately, since the GMRES smoothing is nonlinear
in the starting value, its Fourier symbol can not be derived. In the
following, we will give some explanations for the performance of
GMRES smoothing from the numerical point of view.
\begin{figure}[!htbp]
\centering
    \includegraphics[width=2.4in]{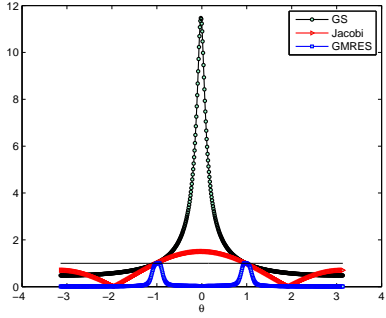}
    \caption{\small Amplification factor of GS, Jacobi and GMRES smoothing for $t=1$.  }\label{fig-smooth-gmres}
\end{figure}

For the ease of presentation, we restrict ourselves to the one
dimensional Helmholtz equation on an interval $(0,10)$ with
homogeneous Dirichlet boundary conditions. For $\kappa=200$, we
apply the grid with mesh size $h=0.005$, i.e., $t=1$. Let the vector
$u_0=e^{{\bf i}\theta \Bx/h}$ be an initial choice for smoothing,
where $\Bx=[x_1,\cdots,x_{N-1}],x_k =
x_{k-1}+kh,x_0=0,k=1,\cdots,N-1$, $N=10/h-1$, $\theta \in
(-\pi,\pi]$. We assume that $\BS_h$ is the relaxation iteration
matrix and $u_1=\BS_h u_0$ is the new vector after one step of
smoothing. Then for fixed $\theta$, we obtain the amplification
factor for one step of smoothing $\rho_s(\theta)=\|u_1\| / \|u_0\|$,
where $\|\cdot\|$ stands for the Euclidean norm. We can see from
Figure \ref{fig-smooth-gmres}, the amplification factor of GMRES
relaxation is always smaller than that of GS and Jacobi relaxations.
When the other two smoothers fail, the GMRES relaxation can still
lead to convergence. Hence, we replace them with GMRES smoothing on
coarse grids.

\subsection{Two  and three level local Fourier analysis}
We have briefly characterized the smoothing procedure in the
multilevel algorithm, in this subsection, the influence of coarse
grid correction will be taken into account. We will focus on the LFA
of two level method and concisely mention the three level method.
For simplicity, we consider the two and three level methods without
post-smoothing and with one step of smoothing on each level. Since
the Fourier symbol can not be obtained for GMRES smoothing, we only
consider the two and three level methods with $\omega$-JAC or GS-LEX
relaxation. Then the iteration operator of Algorithm \ref{alg} in
this simple case can be derived as $ (I-T_L)\cdots(I-T_0), $ where
$T_l =\mu_l I_l R_l A_l P_l$, $R_l$ is smoothing operator.

For the two level method, the iteration matrix is given by
\[
\BM_2 = \Big(\BI_1 -  \mu_1(\BI_1 - \BS_1)\Big)\Big(\BI_1 - \mu_0 \BI^1_0(\BA_0)^{-1} \BI^0_1
\BA_1\Big).
\]
Here, for $l\geq 0$, $\BS_l = \BI_l - \BR_l \BA_l$ is
smoothing relaxation matrix, $\BI_l$ with the same size as $\BA_l$
is identity matrix , $\BI^s_l(s>l)$ is prolongation matrix from
level $l$ to $s$, $\BI^s_l(s<l)$ is restriction matrix from level
$l$ to $s$, and $\BR_l$ stands for matrix representation of
smoother $R_l$.

Since every two dimensional subspace (\ref{Etheta}) of
$2h$-harmonics $E^{\theta^0}_{2h_1}$ with $\theta^0 \in
(-\pi/2,\pi/2]$ is left invariant under $\omega$-JAC or GS-LEX
smoothing operator and correction operator, the representation of
two level iteration matrix of $\BM_2$ on $E^{\theta^0}_{2h_1}$ is
given by a $2 \times 2$ matrix as follows:
\begin{eqnarray}
\widetilde{\BM}_2 &=& \left[
\widetilde{\BI}_1 - \mu_1\left(
\widetilde{\BI}_1 -\left[
\begin{array}{c}
    \widetilde{\BS}_1(\theta^0)    \\
      \widetilde{\BS}_1(\theta^1)\\
  \end{array} \right]_D\right)\right] \nn \\
 && \cdot \left[ \widetilde{\BI}_1 - \mu_0
\left[\begin{array}{c}
    \widetilde{\BI}_0^1(\theta^0)   \\
    \widetilde{\BI}_0^1(\theta^1)\\
  \end{array} \right]\widetilde{\BA}_0(2\theta^0)^{-1}\left[\begin{array}{c}
    \widetilde{\BI}_1^0(\theta^0)   \\
    \widetilde{\BI}_1^0(\theta^1)\\
  \end{array} \right]^t\left[ \begin{array}{c}
    \widetilde{\BA}_1(\theta^0)\\
     \widetilde{\BA}_1(\theta^1)\\
  \end{array}  \right]_D
\right].\label{repre-mfc}
\end{eqnarray}
where $\widetilde{\BI}_1$ is $2\times 2$ identity matrix and the
subscript-$_D$ denotes the transformation of a vector into a
diagonal matrix. Then the spectral radius of
$\widetilde{\BM}_2$ for different $\theta^0\in \Theta_{\rm low}$
can be obtained analytically and numerically.
\begin{figure}[!htbp]
\centering
    \includegraphics[width=2.4in]{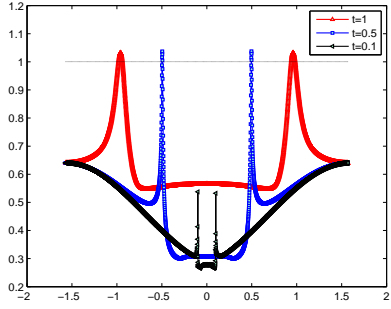}
   \caption{\small $\rho(\widetilde{\BM}_2)$ with $\omega=0.6$ over $\theta^0\in \Theta_{\rm low}$ for $t=1$, $t=0.5$ and $t=0.1$.}\label{fig-jacobi-sl-fc-c}
\end{figure}

Figure \ref{fig-jacobi-sl-fc-c} shows the spectral radius of
$\widetilde{\BM}_2$ with $\omega=0.6$ over $\theta^0\in \Theta_{\rm
low}$ under $\omega$-JAC relaxation for different mesh size $t$ on
the finest grid. If no confusion is possible, we will always denote
$t=$constant for $t = \kappa h_L$ on the finest grid. We observe
that when $t=1$, for most of the frequencies $\theta^0\in
\Theta_{\rm low}$ the amplification factor is smaller than 1. The
amplification factor tends to larger than 1 only under a few
frequencies. Actually, the appearance of such a resonance is caused
by the coarse grid correction and originates from the inversion of
the coarse grid discretization symbol $\widetilde{\BA}_0(2\theta^0)$
in (\ref{repre-mfc}). Based on this reason, the coarsest grid is
chosen to satisfies the mesh condition $\kappa h/p\leq 2$ in our
algorithm. The good performance of two level method for $t=0.5$ and
$t=0.1$ indicates that when the mesh is fine enough to capture the
character of solution, standard smoother works out. We will utilize
the GMRES smoothing when $\kappa h/p\geq 0.5$ and perform weighted
Jacobi or Gauss-Seidel smoothing on those relatively fine grids.

The main idea of the three level analysis is to recursively apply
the previous two level analysis. First, we define the four
dimensional $4h$-harmonics by
\[
E_{4h_2}^{\theta^0} :={\rm span} \{ \varphi_{h_2}(\theta^{00},x) ,
\varphi_{h_2}(\theta^{01},x) , \varphi_{h_2}(\theta^{10},x),
\varphi_{h_2}(\theta^{11},x) \},
\]
where $\theta^{\alpha0} = \frac{\theta^{\alpha}}{2},\theta^{\alpha1}
=
 \frac{\theta^{\alpha}}{2}-{\rm sign}( \frac{\theta^{\alpha}}{2})\pi,\alpha=0,1$.
Similar to the two level method, the iteration matrix can be deduced
to be
\[
\BM_3 = \Big(\BI_2 - \mu_2 (\BI_2 - \BS_2)\Big)
\Big(\BI_2 - \mu_1 \BI^2_1 (\BI_1 - \BS_1)(\BA_1)^{-1}\BI^1_2
\BA_2\Big)\Big(\BI_2 - \mu_0 \BI^2_0(\BA_0)^{-1} \BI^0_2
\BA_2\Big).
\]
 It is easy to see that the three level operator leaves
the space of $4h$-harmonics $E_{4h_2}^{\theta^0} $ invariant (cf.
\cite{WJ04}) for any $\theta^0 \in \Theta_{\rm low}$. This yields a
block diagonal representation of $\BM_3$ with the following $4\times
4$ matrix $\widetilde{\BM}_3$:
\begin{eqnarray}
\widetilde{\BM}_3 &=& \left[ \widetilde{\BI}_2 - \mu_2\left(
\widetilde{\BI}_2 -\widetilde{\BS}_2(\theta)\right)
\right] \nn \\
&&\cdot \left[ \widetilde{\BI}_2 - \mu_1 \widetilde{\BI}_1^2(\theta)
 \left(
\widetilde{\BI}_1 -\left[
\begin{array}{c}
    \widetilde{\BS}_1(\theta^0)    \\
      \widetilde{\BS}_1(\theta^1)\\
  \end{array} \right]_D\right) \left[  \begin{array}{cc}
    \widetilde{\BA}_1(\theta^0)   \\
     \widetilde{\BA}_1(\theta^1)\\
  \end{array}  \right]_D^{-1} (\widetilde{\BI}_2^1(\theta))^t\widetilde{\BA}_2(\theta)
\right]\label{repre-mc-3} \\
&& \cdot \left[ \widetilde{\BI}_2 - \mu_0 \widetilde{\BI}_1^2(\theta)
\widetilde{\BI}_0^1(\theta) \widetilde{\BA}_0(2\theta^0)^{-1}
(\widetilde{\BI}_1^0(\theta))^t (\widetilde{\BI}_2^1(\theta))^t \widetilde{\BA}_2(\theta)
\right],\nn
\end{eqnarray}
where $\widetilde{\BI}_2 $ is $4\times 4$ identity matrix,
$\widetilde{\BS}_2(\theta)  = \left[
\begin{array}{cc}
    \widetilde{\BS}_2(\theta^{00})   \\
    \widetilde{\BS}_2(\theta^{01})    \\
      \widetilde{\BS}_2(\theta^{10})    \\
     \widetilde{\BS}_2(\theta^{11})    \\
  \end{array} \right]_D$, $\widetilde{\BA}_2(\theta)
= \left[  \begin{array}{cc}
    \widetilde{\BA}_2(\theta^{00})   \\
     \widetilde{\BA}_2(\theta^{01})\\
     \widetilde{\BA}_2(\theta^{10})\\
     \widetilde{\BA}_2(\theta^{11})\\
  \end{array}  \right]_D$,
$\widetilde{\BI}_1^2(\theta)  = \left[
\begin{array}{cc}
    \widetilde{\BI}_1^2(\theta^{00})  \quad 0 \\
    \widetilde{\BI}_1^2(\theta^{01})  \quad 0  \\
     0 \quad \widetilde{\BI}_1^2(\theta^{10})    \\
     0 \quad \widetilde{\BI}_1^2(\theta^{11})    \\
  \end{array} \right]$, $\widetilde{\BI}_0^1(\theta) = \left[\begin{array}{c}
    \widetilde{\BI}_0^1(\theta^0)   \\
    \widetilde{\BI}_0^1(\theta^1)\\
  \end{array} \right]$, and $\widetilde{\BI}_2^1(\theta) ,\widetilde{\BI}_1^0(\theta)  $ are defined similarly.

\begin{figure}[htbp]
\centering
    \includegraphics[width=2.4in]{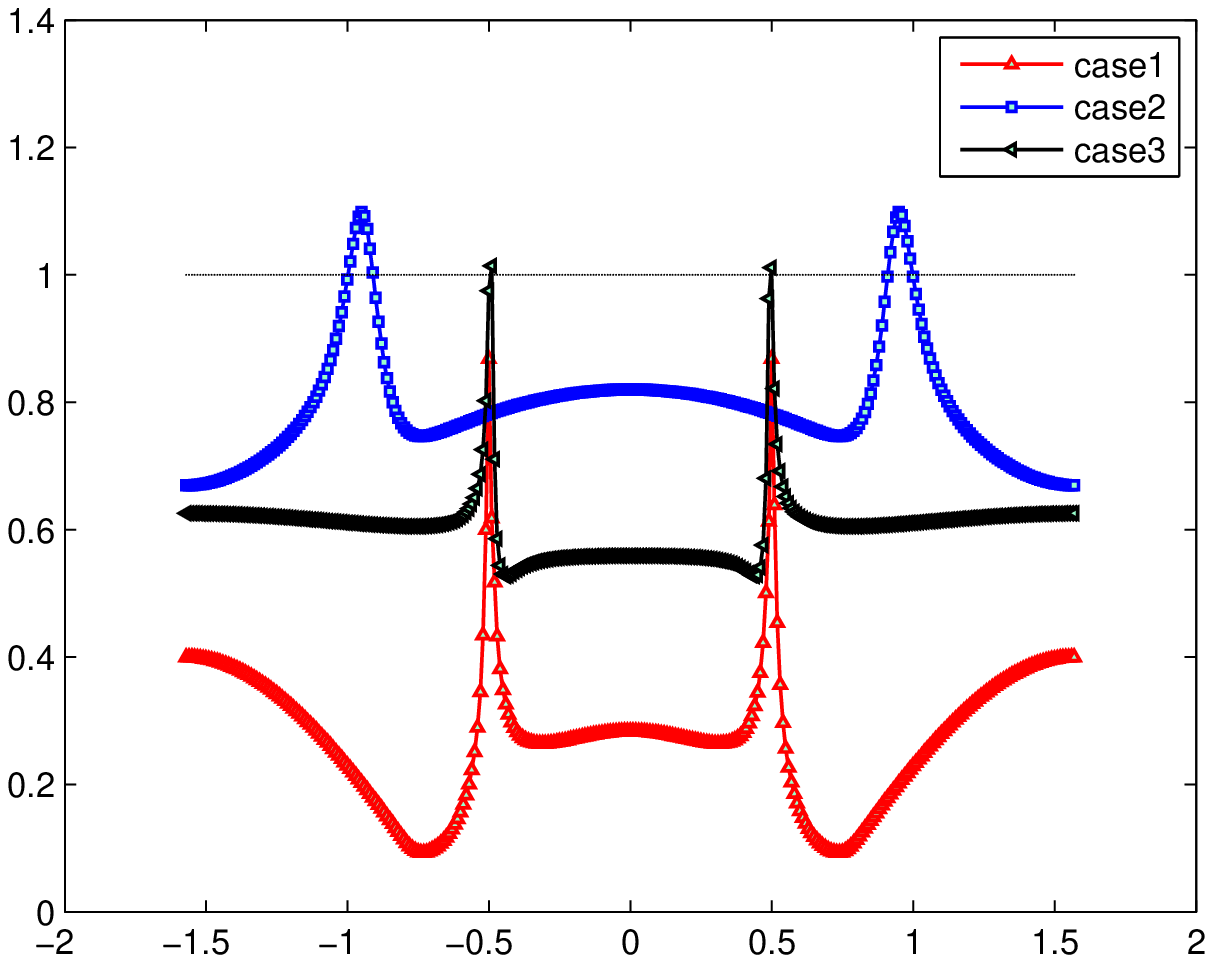}
   \caption{\small Case 1 (t=0.5): $\rho(\widetilde{\BM}_2)$ over $\theta^0 \in \Theta_{\rm low}$, Case 2 (t=0.5) and Case 3 (t=0.25) denote for
   $\rho(\widetilde{\BM}_3)$.
    }\label{fig-jacobi-3d}
\end{figure}

\smallskip

\smallskip

It can be seen from Figure \ref{fig-jacobi-3d}, when the mesh size
on the coarsest grid is determined to be the same, the performance
of two  and three level methods behave similarly, although the two
level method has smaller spectral radius. We also observe that a
coarser initial grid used for $\BM_3$ will deteriorate its
convergence.

\section{Numerical results}\label{sec-numer}
We will present two numerical examples to demonstrate Algorithm \ref{alg}
in two dimension. Our multilevel
algorithm is used as a preconditioner in outer GMRES iterations
(PGMRES). The level $l$ which distinguishes the smoothing strategy
satisfies $\kappa h_l / p \thickapprox 0.5$. We always use
Gauss-Seidel relaxation when $\kappa h_l/ p < 0.5$ and GMRES
relaxation otherwise in Algorithm \ref{alg}. The smoothing step
$\{m_i\}^4_{i=1}$ is chosen as two if there is no any annotation. At
the $l$-th level, the discrete problem is ${\bf A}_l {\bf u}_l =
{\bf F}_l$. Let $ {\bf r}^n_l = {\bf F}_l - {\bf A}_l {\bf u}^n_l$
be the residual with respect to the $n$-th iteration. The PGMRES
algorithm stops when
\begin{equation}\label{e:6}
\|{\bf r}^n_l\|/\|{\bf r}^0_l\| \leq 10^{-6} ,\nn
\end{equation}
where $\|v\|$ is the $L^2$ norm of the vector $v$.
The number of iteration steps required to achieve the desired accuracy is denoted by {\bf iter}.

\begin{exm}
{\rm We consider a two dimensional Helmholtz equation with the first order
absorbing boundary condition (cf. \cite{Wu09,Wu12}):
\begin{align*}
-\Delta u - \kappa^2 u &= f:= \frac{\sin(\kappa r)}{r}\quad {\rm in}
\
\Omega,  \\
\frac{\partial u}{\partial n} + {\ii} \kappa u &= g \quad {\rm on} \
\partial \Omega.
\end{align*}
Here $\Omega$ is a unit square with center $(0,0)$ and $g$ is chosen
such that the exact solution is
\[
u = \frac{\cos(\kappa r)}{\kappa} - \frac{\cos \kappa+ \ii \sin
\kappa}{\kappa(J_0(\kappa)+{\ii} J_1(\kappa))} J_0 (\kappa r),
\]
where $J_{\nu}(z)$ are Bessel functions of the first kind.
}

\begin{table}[!h]
\caption{\footnotesize Iteration  number of PGMRES based on Algorithm
\ref{alg} for the HDG-P1 in the
cases $\kappa=50,100,200$ with coarsest grid size $\kappa h_0/p\approx 2$.}\label{ex1-1}
\begin{center}
\footnotesize
\begin{tabular}{|c|c|c|ccc|}
\hline
\multirow{4}{*}{$\kappa=50$} & \multicolumn{2}{|c|}{Level}      &   3    &  4     &  5  \\
\cline{2-6}
 & \multicolumn{2}{|c|}{DOFs}                                  & 98816   & 394240 & 1574912   \\
\cline{2-6}
 & \multicolumn{2}{|c|}{iter (P1)}                                   &   20    &  16     &   15      \\
 \hline \hline

\multirow{4}{*}{$\kappa=100$} & \multicolumn{2}{|c|}{Level}       &   3    &  4     &  5  \\
\cline{2-6}
 & \multicolumn{2}{|c|}{DOFs}                                  & 394240 & 1574912 & 6295552 \\
\cline{2-6}
 & \multicolumn{2}{|c|}{iter (P1)}                               &  30  &  20 &  19 \\
\hline \hline

\multirow{4}{*}{$\kappa=200$} & \multicolumn{2}{|c|}{Level}   &    2    &   3    &  4      \\
\cline{2-6}
 & \multicolumn{2}{|c|}{DOFs}                                & 394240 & 1574912 & 6295552      \\
\cline{2-6}
 & \multicolumn{2}{|c|}{iter (P1)}                     &  76    &   54   &  35    \\
\hline

\end{tabular}
\end{center}
\end{table}

\begin{table}[!h]
\caption{\footnotesize Iteration number  of PGMRES based on Algorithm
\ref{alg} for the HDG-P2 in the
cases $\kappa=50,200,360$ with coarsest grid size $\kappa h_0/p\approx 2$.}\label{ex1-2}
\begin{center}
\footnotesize
\begin{tabular}{|c|c|c|ccc|}
\hline
\multirow{4}{*}{$\kappa=50$} & \multicolumn{2}{|c|}{Level}      &   3    &  4     &  5  \\
\cline{2-6}
 & \multicolumn{2}{|c|}{DOFs}                                 & 37248 & 148224  & 591360    \\
\cline{2-6}
 & \multicolumn{2}{|c|}{iter (P1)}                                   & 11      &  10    &  9    \\
 \hline \hline

\multirow{4}{*}{$\kappa=200$} & \multicolumn{2}{|c|}{Level}       &   2    &  3     &  4  \\
\cline{2-6}
 & \multicolumn{2}{|c|}{DOFs}                                 & 148224 & 591360 & 2362368  \\
\cline{2-6}
 & \multicolumn{2}{|c|}{iter (P1)}                             &  30 &  29  &  24    \\
\hline \hline

\multirow{4}{*}{$\kappa=360$} & \multicolumn{2}{|c|}{Level}   &    2    &   3    &  4      \\
\cline{2-6}
 & \multicolumn{2}{|c|}{DOFs}                                 &    591360    &  2362368     &   9443328     \\
\cline{2-6}
 & \multicolumn{2}{|c|}{iter (P1)}                     &  30    &   31   &   25 \\
\hline

\end{tabular}
\end{center}
\end{table}

\end{exm}
\smallskip

\begin{figure}[!h]
\centering
    \includegraphics[width=2.4in]{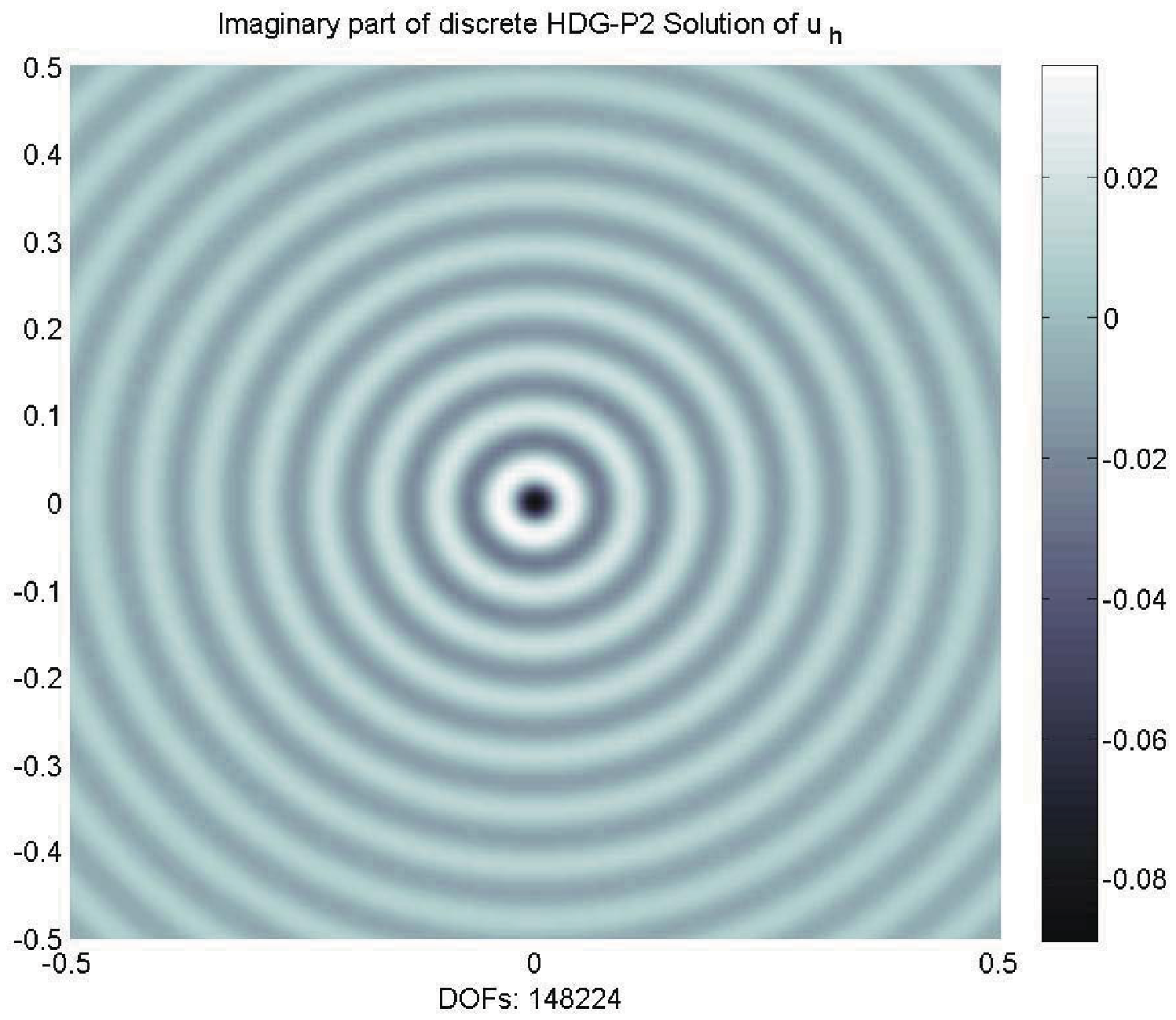}
    \includegraphics[width=2.4in]{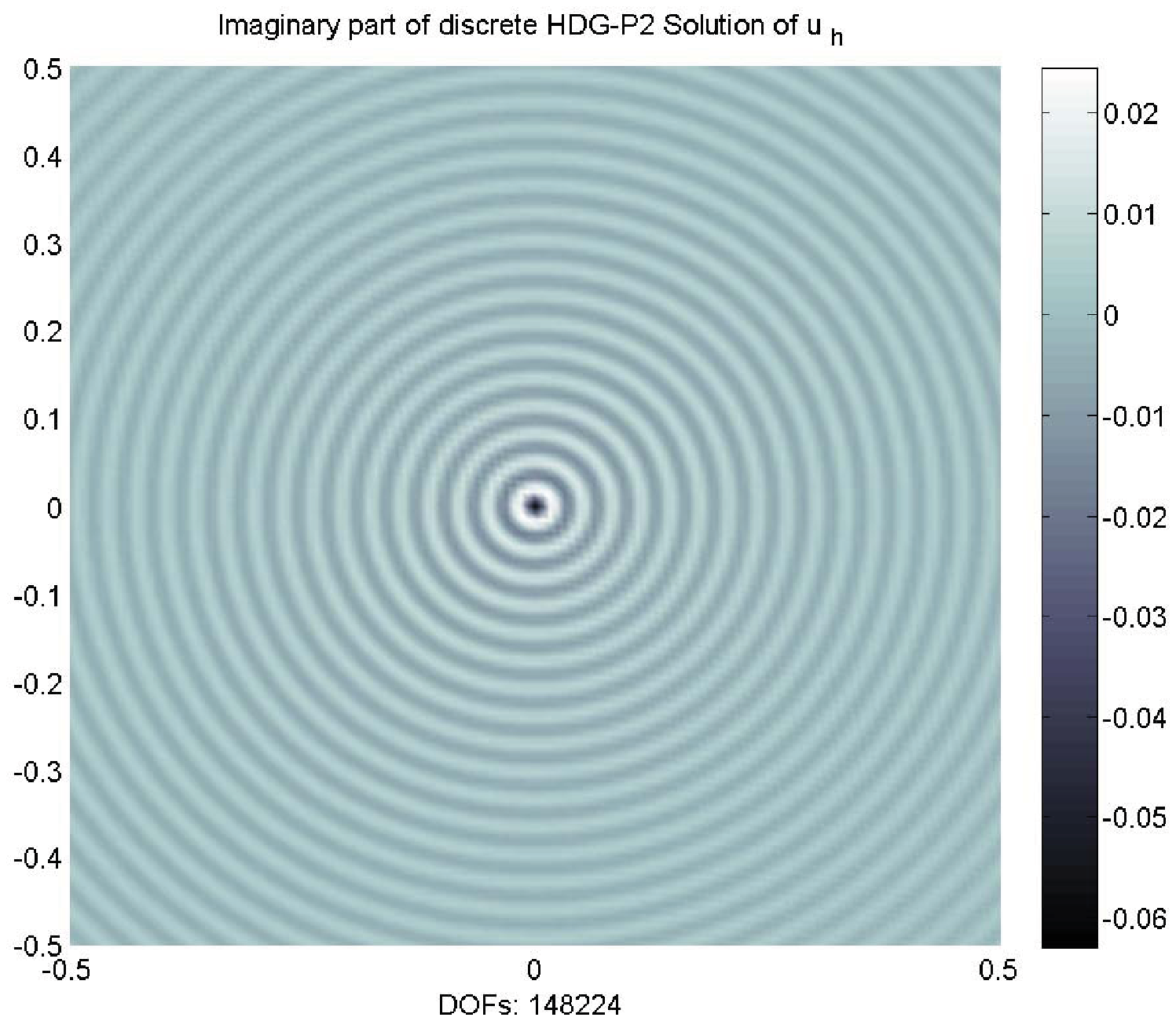}
    \caption{\footnotesize Surface plot of imaginary part of discrete HDG-P2 solutions for $\kappa=100$ (left) and $\kappa=200$ (right)
    on the grid with mesh condition $\kappa h / p \approx 0.55$ (left) and $1.1$ (right). }\label{ex1-fig}
\end{figure}

In this example, the coarsest level of multilevel method is chosen
to satisfy $\kappa h_0/p \thickapprox 2$ for $\kappa\leq 360$. For
$400 \leq \kappa \leq 700$, we choose the coarsest grid condition
with the same mesh size $h_0$ such that $\kappa h_0/p\thickapprox
1.1 \thicksim 1.9$. We can observe from Table \ref{ex1-1} and Table
\ref{ex1-2} that the iteration number is mesh independent for fixed
$\kappa$, and it increases mildly with large wave number. We note
that for the piecewise linear polynomial, HDG method does not have
the advantage of saving degrees of freedom, while in the case of
piecewise quadratic polynomial, the ratio between the number of
degrees of freedom for HDG  and standard DG method is about $\frac{3}{4}$. And
the higher the polynomial degree is, the lower the ratio is. Hence
we focus on HDG-P2 for the performance of our algorithm in this
example.

Figure \ref{ex1-fig} displays the surface plot of imaginary part of
discrete HDG-P2 solutions for $\kappa=100,200$ on the grid with mesh
condition $\kappa h / p \approx 0.55$ and $1.1$ respectively.
Indeed, the discrete solutions have correct shapes and amplitudes as
the exact solutions. We also test the performance of PGMRES based on
Algorithm \ref{alg} with different smoothing steps. Table
\ref{ex1-3} shows that when it takes two steps of smoothing the
iteration number is much small with respect to one smoothing step.
But the advantage of reducing the iteration number by adding more
smoothing steps is deteriorating, and more steps of GMRES smoothing
requires more memory to store data in the computation. Hence, in the
following we will use two smoothing steps
in Algorithm \ref{alg}.

\begin{table}[!h]\footnotesize
\caption{\footnotesize Iteration number of PGMRES based on Algorithm
\ref{alg} for HDG-P2 with
different  smoothing steps
($m_i=m,i=1,\cdots,4$, $\kappa=100$).}\label{ex1-3}
\begin{center}
\begin{tabular}{|c|c|ccc|}
\hline
\multirow{5}{*}{HDG-P2} & Level  & 3 & 4 & 5 \\
\cline{2-5}
& DOFs   & 148224  & 591360  & 2362368 \\
\cline{2-5}
& iter ($m=1$)   & 30  & 27  & 21 \\
\cline{2-5}
& iter ($m=2$)    & 18  & 15  & 13  \\
\cline{2-5}
& iter ($m=3$)     &  15 & 13   &  11 \\
\hline
\end{tabular}
\end{center}
\end{table}

\begin{table}[!h]
\caption{\footnotesize Iteration number of PGMRES based on Algorithm
\ref{alg} for HDG-P2 in the cases $\kappa=400,500,600,700$.}\label{ex1-4}
\begin{center}
\footnotesize
\begin{tabular}{|c|c|c|cc|}
\hline
\multirow{6}{*}{HDG-P2} & \multicolumn{2}{|c|}{Level}      &   2    &  3     \\
\cline{2-5}
 & \multicolumn{2}{|c|}{DOFs}                                  &  2362368 & 9443328  \\
\cline{2-5}
 & \multicolumn{2}{|c|}{iter ($\kappa=400$)}                                   &   11     &  11          \\
 \cline{2-5}
 & \multicolumn{2}{|c|}{iter ($\kappa=500$)}                                    &   16   &   16    \\
 \cline{2-5}
 & \multicolumn{2}{|c|}{iter ($\kappa=600$)}                                    &   26   &   27    \\
 \cline{2-5}
 & \multicolumn{2}{|c|}{iter ($\kappa=700$)}                                    &   45   &   50   \\
\hline
\end{tabular}
\end{center}
\end{table}

Table \ref{ex1-4} shows the iteration number of PGMRES based on
Algorithm \ref{alg} for the cases $\kappa=400,500,600,700$ with the
same coarsest grid $h_0\approx 0.00552$ such that $\kappa h_0/p\thickapprox 1.1 \thicksim
1.9$, we can see that
the  iteration number is still stable and acceptable.

\bigskip

\begin{exm}
{\rm We consider a cave model in a unit square domain with center
$(0,0)$. Figure \ref{ex2-fig} shows the computational domain and the
variation of wave number in different subdomains which are indicated
by different colors.

\begin{figure}[!h]
\centering
    \includegraphics[width=2.4in]{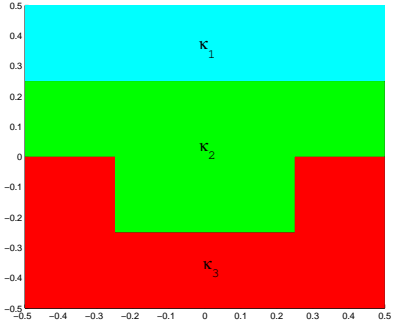}
    \caption{\footnotesize The computational domain of cave problem with different wave number indicated. }\label{ex2-fig}
\end{figure}

We denote by $\kappa_3=q_2 \kappa_2=q_1 \kappa_1$. The Robin boundary
condition (\ref{HE}) is set to be $g=0$ and the external force
$f(x)$ in (\ref{HSE}) is a narrow Gaussian point source (cf.
\cite{Engquist2}) located at the center $(0,0)$:
\[
f(x_1,x_2) = \frac{1}{{\ii \kappa}}e^{ -(\frac{4\kappa}{\pi})^2 (
x_1^2 + x_2^2 )  }.
\]

\begin{table}[!h]
\caption{\footnotesize Iteration number of PGMRES based on Algorithm
\ref{alg} for HDG-P1, HDG-P2 and HDG-P3 discretizations for the case
$\kappa_3=200,q_2=2,q_1=3$.}\label{ex2-1}
\begin{center}
\footnotesize
\begin{tabular}{|c|c|c|ccc|}
\hline
\multirow{4}{*}{HDG-P1} & \multicolumn{2}{|c|}{Level}      &   2    &   3    &  4      \\
\cline{2-6}
 & \multicolumn{2}{|c|}{DOFs}                                  &  221952   & 886272  &  3542016  \\
\cline{2-6}
 & \multicolumn{2}{|c|}{iter}                                   &   166    &   129     &     54    \\
 \hline \hline

\multirow{4}{*}{HDG-P2} & \multicolumn{2}{|c|}{Level}       &   2    &   3    &  4      \\
\cline{2-6}
 & \multicolumn{2}{|c|}{DOFs}                                  & 83520 & 332928  &  1329408 \\
\cline{2-6}
 & \multicolumn{2}{|c|}{iter}                               &   53 & 66  &  54  \\
\hline \hline

\multirow{4}{*}{HDG-P3} & \multicolumn{2}{|c|}{Level}   &    2    &   3    &  4      \\
\cline{2-6}
 & \multicolumn{2}{|c|}{DOFs}                                &  197632 &  788480 &   3149824     \\
\cline{2-6}
 & \multicolumn{2}{|c|}{iter}                     &  23   &   26    &  26   \\
\hline

\end{tabular}
\end{center}
\end{table}

For this problem we firstly test the performance of our multilevel
method for HDG method with different polynomial order
approximations. In Table \ref{ex2-1}, the iteration number for
HDG-P1 and HDG-P2 are based on the coarsest grid condition with
$\kappa_3 h_0 / p \thickapprox 2.95$. We can see that the multilevel method
is more stable when the higher polynomial order approximation is applied. But for HDG-P3, the iteration number will be more than 200
with the above coarsest grid condition. Thus, we utilize the
coarsest grid condition $\kappa_3 h_0 / p \thickapprox 1.47$ for
HDG-P3, then the convergence of PGMRES becomes stable. Comparing the
iteration number for HDG-P1 and HDG-P3, one can also observe that
when the degrees of freedom are similar on each level, the
convergence of PGMRES is more stable for higher polynomial order
approximation. In the following, we focus on the performance of
HDG-P2.

\begin{figure}[!h]
\centering
    \includegraphics[width=2.4in,height=2in]{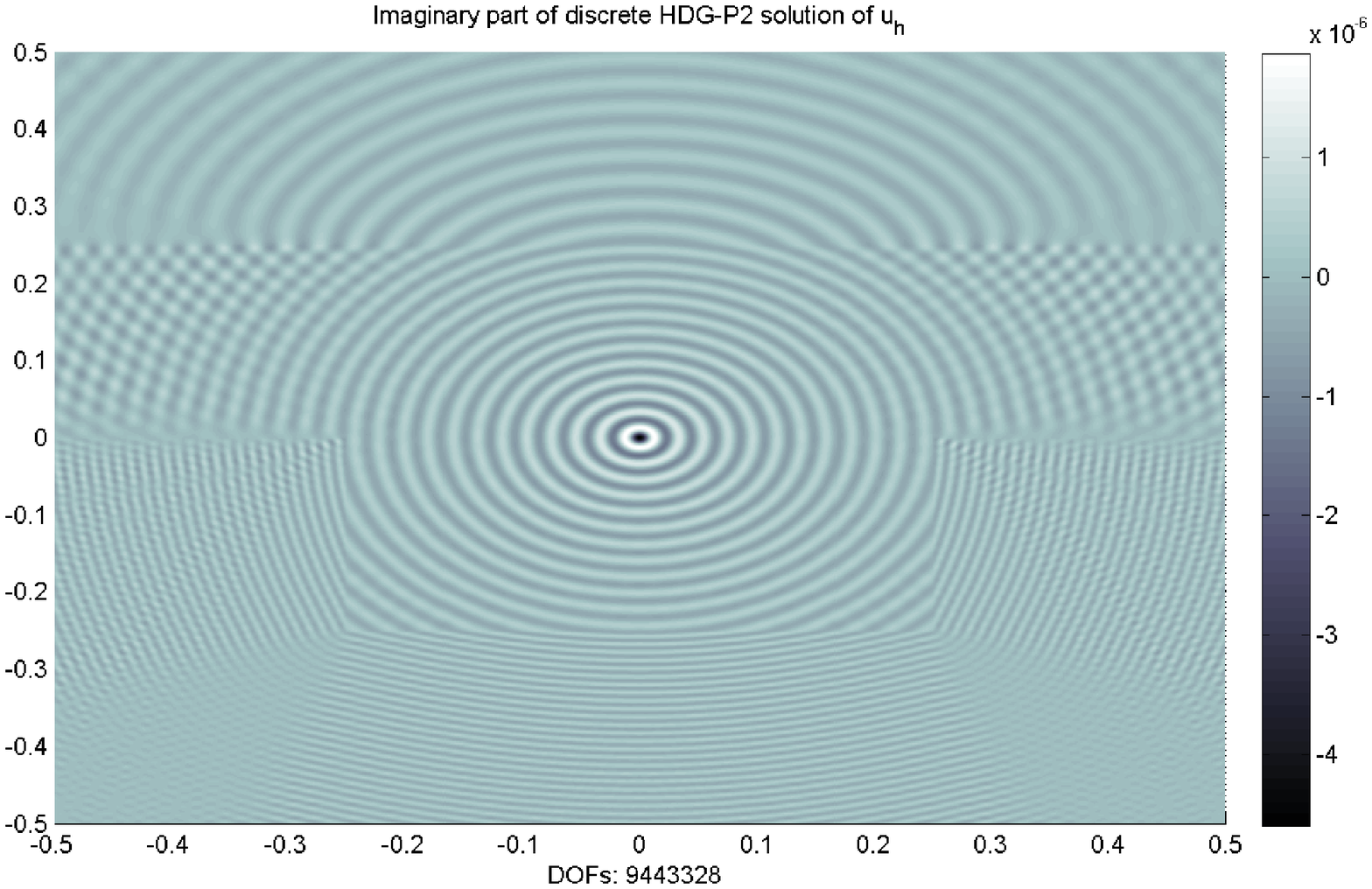}
    \includegraphics[width=2.4in,,height=2in]{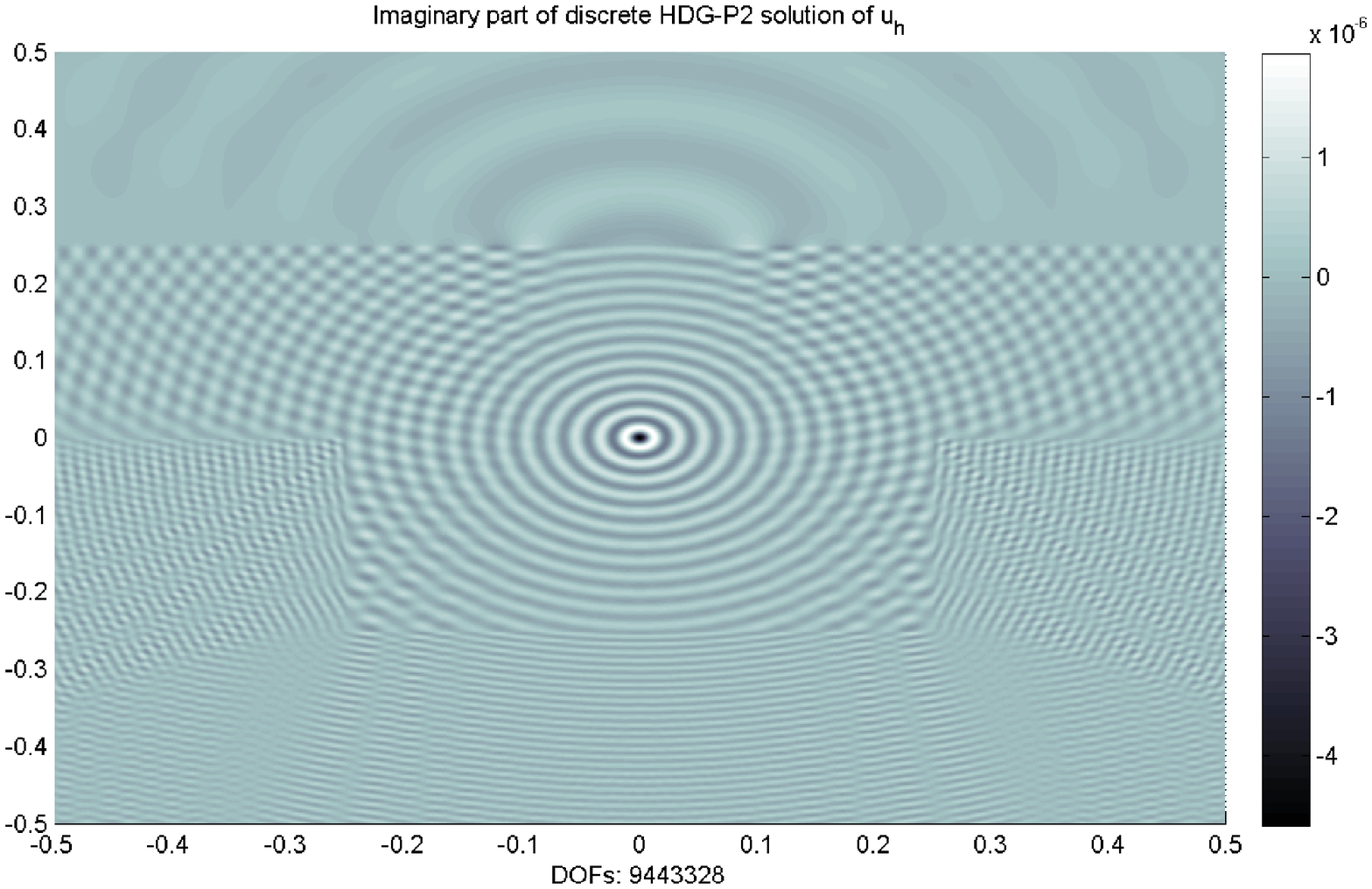}
    \caption{\footnotesize Surface plot of imaginary part of discrete HDG-P2 solutions for $\kappa_3=600, q_2=2, q_1=3$ (left) and $\kappa_3=600, q_2=2, q_1=10$ (right)
    on the grid with mesh condition $\kappa_3 h / p \approx 0.4$. }\label{ex2-fig}
\end{figure}

Figure \ref{ex2-fig} displays the surface plot of imaginary part of
discrete HDG-P2 solutions $u_h$ for $\kappa_3=600$ with $q_2=2,q_1=3$
and $q_1=10$ on the grid with mesh condition $\kappa_3 h / p \approx
0.4$. The iteration number of PGMRES based on Algorithm \ref{alg}
for different $\kappa_3,q_2$ and $q_1$ are listed in Table \ref{ex2-2}.
The larger jump of wave numbers between different subdomain will
deteriorate the convergence of the algorithm. For instance, the
iteration number for the case $q_1=10$ are much more than that for
$q_1=3$. But for fixed $\kappa_3,q_2$ and $q_1$, the iteration number are
robust on different levels.

\begin{table}[h]
\caption{\footnotesize Iteration number of PGMRES based on Algorithm
\ref{alg} for HDG-P2 discretizations for the cases $\kappa_3=400$
$(q_2=2,q_1=3,q_1=10)$ and $\kappa_3=600$ $(q_2=2,q_1=3,q_1=10)$.}\label{ex2-2}
\begin{center}
\footnotesize
\begin{tabular}{|c|c|c|cc|}
\hline
\multirow{4}{*}{$\kappa=400$} & \multicolumn{2}{|c|}{Level}      &   2     &     3    \\
\cline{2-5}
 & \multicolumn{2}{|c|}{DOFs}                                  &  2362368  &  9443328     \\
\cline{2-5}
 & \multicolumn{2}{|c|}{iter $(q_2=2,q_1=3)$}                                   &   14     &       14    \\
 \cline{2-5}
 & \multicolumn{2}{|c|}{iter $(q_2=2,q_1=10)$}                                    &  28    &     27  \\
\hline \hline

\multirow{4}{*}{$\kappa=600$} & \multicolumn{2}{|c|}{Level}       &    2    &    3    \\
\cline{2-5}
 & \multicolumn{2}{|c|}{DOFs}                                 &  2362368  &  9443328    \\
\cline{2-5}
 & \multicolumn{2}{|c|}{iter $(q_2=2,q_1=3)$}                               &  24  &   26    \\
\cline{2-5}
 & \multicolumn{2}{|c|}{iter $(q_2=2,q_1=10)$}                                &  45  & 46    \\
\hline
\end{tabular}
\end{center}
\end{table}

}

\end{exm}


\bigskip


\begin{thebibliography}{bib}

\bibitem{Adams}
R. Adams, Sobolev Spaces, Academic Press, New York, 1975.

\bibitem{Brandt77}
A. Brandt, Multi-level adaptive solutions to boundary-value
problems, Math. Comp., 31 (1977), pp. 333--390.

\bibitem{Brandt97}
A. Brandt and I. Livshits, Wave-ray multigrid method for
standing wave equations, Electron. Trans. Numer. Anal., 6 (1997),
pp. 162--181.

\bibitem{Erik}
E. Burman and A. Ern, Continuous interior penalty $hp$-finite element methods for
advection and advection-diffusion equations, Math. Comp., 76 (2007), pp. 1119--1140.

\bibitem{HDG}
H. Chen, P. Lu and X. Xu, A hybridizable discontinuous Galerkin
method for the Helmholtz equation with high wave number, SIAM J. Numer. Anal., to appear,  2013.

\bibitem{chen}
H. Chen, H. Wu and X. Xu, Multilevel preconditioner with stable
coarse grid corrections for the Helmholtz equation, submitted,
2013.

\bibitem{Cockburn}
B. Cockburn, J. Gopalakrishnan and R. Lazarov, Unified hybridization of discontinuous Galerkin, mixed, and continuous Galerkin methods for second order elliptic problems, SIAM J. Numer. Anal.,
47 (2009), pp. 1319--1365.

\bibitem{Cockburn2}
B. Cockburn, J. Gopalakrishnan and F.J. Sayas, A projection-based error analysis of HDG methods, Math Comp., 79 (2010), pp. 1351--1367.

\bibitem{EEO01}
H.C. Elman, O.G. Ernst, and D.P. O'Leary, A multigrid method
enhanced by Krylov subspace iteration for discrete Helmholtz
equations, SIAM J. Sci. Comput., 23 (2001), pp. 1291--1315.


\bibitem{Engquist1}
B. Engquist and L. Ying, Sweeping preconditioner for the
Helmholtz equation: hierarchical matrix representation, Comm. Pure
Appl. Math., 64 (2011), pp. 697--735.

\bibitem{Engquist2}
B. Engquist and L. Ying, Sweeping preconditioner for the
Helmholtz equation: moving perfectly matched layers, Multiscale Model. Simul.,
9 (2011), pp. 686--710.


\bibitem{Erlangga08}
Y.A. Erlangga, Advances in iterative methods and preconditioners for
the Helmholtz equation, Arch. Comput. Methods Eng., 15 (2008), pp.
37--66.

\bibitem{Erlangga04}
Y.A. Erlangga, C. Vuik, and C.W. Oosterlee, On a class of
preconditioners for solving the Helmholtz equation, Appl. Numer.
Math., 50 (2004), pp. 409--425.

\bibitem{Erlangga06}
Y.A. Erlangga, C.W. Oosterlee, and C. Vuik, A novel multigrid
based preconditioner for heterogeneous Helmholtz problems, SIAM J.
Sci. Comput., 27 (2006), pp. 1471--1492.

\bibitem{Ernst11}
O.G. Ernst and M.J. Gander, Why it is difficult to solve Helmholtz
problems with classical iterative methods, in: I. Graham, T. Hou, O. Lakkis, R. Scheichl (Eds.), Numerical Analysis of Multiscale Problems, Springer, 2011.


\bibitem{Wuhp}
X. Feng and H. Wu, $hp$-discontinuous Galerkin methods for the Helmholtz equation with large wave number, Math. Comp., {80} (2011), pp. 1997--2024.

\bibitem{Wu09}
X. Feng and H. Wu, Discontinuous Galerkin methods for the Helmholtz
equation with large wave number, SIAM J. Numer. Anal., { 47}
(2009), pp. 2872--2896.

\bibitem{JG}
J. Gopalakrishnan, A Schwarz preconditioner for a hybridized
mixed method, Comput. Methods Appl. Math., 3 (2003), pp. 116--134.

\bibitem{J-S}
J. Gopalakrishnan and S. Tan, A convergent multigrid cycle for the hybridized
mixed method, Numer. Linear Algebra Appl., 16 (2009), pp. 689--714.

\bibitem{Livshits}
I. Livshits and A. Brandt, Accuracy properties of the wave-ray
multigrid algorithm for Helmholtz equations, SIAM J. Sci. Comput.,
28 (2006), pp. 1228--1251.

\bibitem{peter}
P. Monk, Finite Element Methods for Maxwell's Equations, Oxford University Press, Oxford,
UK, 2003.

\bibitem{WJ04}
R. Wienands and W. Joppich, Practical Fourier Analysis for Multigrid
Methods, Chapman \& Hall/CRC, London, 2004.

\bibitem{Wu12}
H. Wu, Pre-asymptotic error analysis of CIP-FEM and FEM for
Helmholtz equation with high wave number. Part I: Linear version, IMA J. Numer. Anal., to appear, 2013.


\bibitem{Wu12-hp}
L. Zhu and H. Wu, Pre-asymptotic error analysis of CIP-FEM and FEM
for Helmholtz equation with high wave number. Part {II}: $hp$
version, SIAM J. Numer. Anal., to appear, 2013.


\end{thebibliography}
\end{document}